\documentclass[reqno]{amsproc}
\usepackage{amsmath}
\usepackage{amssymb}
\usepackage{graphicx}
\usepackage{multirow}
\DeclareGraphicsExtensions{.png,.pdf,.eps}
\usepackage[all,2cell]{xy}
\usepackage{tikz}
\usepackage{pgf}
\usepackage{enumerate}
\usepackage{mathdots}
\usepackage{enumitem}
\usepackage[customcolors]{hf-tikz}

\usetikzlibrary{cd}
\usepackage{tikz-cd}
\usepackage{hyperref}

\newtheorem{theorem}{Theorem}[section]
\newtheorem{lemma}[theorem]{Lemma}
\newtheorem{corollary}[theorem]{Corollary}
\newtheorem{proposition}[theorem]{Proposition}

\theoremstyle{definition}

\newtheorem{definition}[theorem]{Definition}
\newtheorem{notation}[theorem]{Notation}
\newtheorem{setup}[theorem]{Setup}

\newtheorem{remark}[theorem]{Remark}

\newtheorem{set}[theorem]{Setup}
\newtheorem{example}[theorem]{Example}

\title[Chow Quotients]{Chow Quotients of $\CC^*$-actions}

\author[Occhetta]{Gianluca Occhetta}
\address{Dipartimento di Matematica, Universit\`a degli Studi di Trento, via
Sommarive 14 I-38123 Povo di Trento (TN), Italy}
 
\email{gianluca.occhetta@unitn.it, eduardo.solaconde@unitn.it}

\author[Romano]{Eleonora A. Romano}
\address{Dipartimento di Matematica, Dipartimento di Eccellenza 2023-2027, Universit\`a degli Studi di Genova, via Dodecaneso 35, I-16146, Genova (GE), Italy}
\email{eleonoraanna.romano@unige.it}

\author[Sol\'a Conde]{Luis E. Sol\'a Conde}

\author[Wi\'sniewski]{Jaros\l{}aw A. Wi\'sniewski}
\address{Instytut Matematyki UW, Banacha 2, 02-097 Warszawa, Poland}
\email{J.Wisniewski@uw.edu.pl}


\thanks{First, third and fourth author supported by the Department of Mathematics of the University of Trento. First, second and third author partially supported by INdAM--GNSAGA. Second author supported in part by the MIUR Excellence Department Project awarded to Dipartimento di Matematica, Universit\`a di Genova, CUP D33C23001110001. % , PRIN project ``Geometria delle variet\`a algebriche''. 
Fourth author supported by Polish National Science Center project 2022/47/B/ST1/01896.}

\usepackage[textsize=tiny]{todonotes}

\usepackage{enumitem}
\newcommand\ignore[1]{}

\DeclareMathOperator{\HH}{H}
\newcommand\TT{{\mathbb{T}}}
\newcommand\CC{{\mathbb{C}}}
\newcommand\PP{{\mathbb{P}}}

\newcommand\QQ{{\mathbb{Q}}}

\def\B{{\mathbb B}}
\def\C{{\mathbb C}}

\def\G{{\mathbb G}}

\def\P{{\mathbb P}}
\def\Q{{\mathbb Q}}
\def\R{{\mathbb R}}
\def\Z{{\mathbb Z}}

\def\cB{{\mathcal B}}

\def\cE{{\mathcal E}}

\def\cI{{\mathcal I}}

\def\cL{{\mathcal L}}
\def\cM{{\mathcal M}}
\def\cN{{\mathcal{N}}}
\def\cO{{\mathcal{O}}}

\def\cU{{\mathcal U}}

\def\cY{{\mathcal Y}}
\def\Q{{\mathbb{Q}}}
\def\G{{\mathbb{G}}}

\def\operatorname#1{\mathop{\rm #1}\nolimits}

\def\DA{{\rm A}}

\def\Proj{\operatorname{Proj}}

\def\Chow{\operatorname{Chow}}
\def\Ext{\operatorname{Ext}}
\def\Exc{\operatorname{Exc}}
\def\Hilb{\operatorname{Hilb}}
\def\Hom{\operatorname{Hom}}

\def\Pic{\operatorname{Pic}}
\def\Hom{\operatorname{Hom}}

\def\codim{\operatorname{codim}}

\def\id{\operatorname{id}}

\def\Int{\operatorname{Int}}

\def\rk{\operatorname{rk}}

\def\Bs{\operatorname{Bs}}

\def\det{\operatorname{det}}

\def\alg{\operatorname{alg}}

\def\Nef{{\operatorname{Nef}}}

\def\NU{{\operatorname{N^1}}}

\def\Mov{{\operatorname{Mov}}}

\def\GX{\mathcal{G}\!X}
\def\CX{\mathcal{C}\!X}
\def\HX{\mathcal{H}\!X}

\def\ss{\operatorname{ss}}

\newcommand{\pb}{\ar@{}[dr]|{\text{\pigpenfont J}}}
\def\ol{\overline}

\newcommand{\shse}[3]{0 ~\ra ~#1~ \lra ~#2~ \lra ~#3~ \ra~ 0}
\makeatletter
\newcommand{\xleftrightarrow}[2][]{\ext@arrow 3359\leftrightarrowfill@{#1}{#2}}
\makeatother
\newcommand{\xdasharrow}[2][->]{
\tikz[baseline=-\the\dimexpr\fontdimen22\textfont2\relax]{
\node[anchor=south,font=\scriptsize, inner ysep=1.5pt,outer xsep=2.2pt](x){#2};
\draw[shorten <=3.4pt,shorten >=3.4pt,dashed,#1](x.south west)--(x.south east);
}}
\newcommand{\hooklongrightarrow}{\lhook\joinrel\longrightarrow}

\newcommand\m{{\mathfrak m}}

\newcommand\ra{{\ \rightarrow\ }}
\newcommand\lra{\longrightarrow}

\newcommand\iso{{\ \cong\ }}

\def\Mo{\operatorname{\hspace{0cm}M}}

\def\prul{s}
\def\prur{d}
\begin{document}
\begin{abstract}
Given an action of the one-dimensional torus on a projective variety, the associated Chow quotient arises as a natural parameter space of invariant $1$-cycles, which dominates the GIT quotients of the variety.  
In this paper we explore the relation between the Chow and the GIT quotients of a variety, showing how to construct explicitly the former upon the latter via successive blowups under suitable assumptions. We also discuss conditions for the smoothness of the Chow quotient, and present some examples in which it is singular.
\end{abstract}

\maketitle
\tableofcontents
%%%%%%%%%%%%%%%%%%%%%%%%%%%%%%%%%%%%%%%%%%%%%%%%%%%%%%%%%%%%%%%%%%%%%%%%%%%%%%%
%%%%%%%%%%%%%%%%%%%%%%%%%%%%%%%%%%%%%%%%%%%%%%%%%%%%%%%%%%%%%%%%%%%%%%%%%%%%%%%
\section{Introduction}\label{sec:intro}

The relation of birational geometry with the theory of quotients of $\CC^*$-actions has been acknowledged since the early days of  Mori Theory and the Minimal Model Program; for early expositions on this matter see the contributions by Thaddeus and Reid \cite{Thaddeus1996}, \cite{ReidFlip}. W{\l}odarczyk in  \cite{Wlodarczyk} used the notion of birational cobordism, constructed by Morelli in the toric case \cite{Morelli}, to prove the Weak Factorization Conjecture,  
which asserts that given a birational map of smooth projective varieties $\begin{tikzcd}[cramped]Y\ar[r,<->, dashed]&Y'\end{tikzcd}$, there exists a sequence of blowups and blowdowns along smooth centers (or, possibly, isomorphisms) resolving this map  (see also \cite{AKMW} and \cite{Wlodarczyk2009} for a broader view on this topic). In W{\l}odarczyk's proof a variety with a $\CC^*$-action is constructed, the cobordism $X$, whose two distinguished quotients are the varieties $Y$ and $Y'$; the crucial step in the argument is about elementary cobordisms associated to flips. 

Recently, the relation between birational geometry and $\C^*$-actions has been reconsidered under the viewpoint of Mumford's Geometric Invariant Theory (GIT, see \cite{MFK}), Variation of GIT (VGIT, see \cite{DolgachevHu}), and the Bia{\l}ynicki-Birula decomposition (see \cite{BB,BBS}). This viewpoints allow to go deeper into that relation: projective varieties endowed with $\C^*$-actions can be seen as geometric realizations of birational transformations, and linearizations of these actions provide natural factorizations of the given transformations (cf. \cite{WORS4, WORS5, WORS3,WORS1, BRUS}). 

Given a $\C^*$-action on a projective variety $X$ and a linearization of an ample line bundle $L$ on $X$, the geometric quotients of the action can be described in terms of the weights $a_0, \dots, a_r$, which we associate to fixed point components of the action; we will accordingly denote them by  $\GX_{i,i+1}$, for $i=0, \dots, r-1$ (see Section \ref{sssec:quotients}). The positive integer $r$  is called the criticality of the action. In particular,  the two extremal quotients, $\GX_{0,1}$ and $\GX_{r-1,r}$ parametrize $1$-dimensional orbits converging to points in the sink and the source of the action, respectively.

The action defines birational maps among the geometric quotients, in particular, between the two extremal ones; the birational map between them has a natural factorization via the intermediate quotients: 
\[
\begin{tikzcd}[
  column sep={2.8em,between origins},
  row sep={3.em,between origins}]
  \GX_{0,1}\ar[rrr, dashed,"\mbox{\tiny bir.}"]&&&\GX_{i,i+1}\ar[rrrrrr, dashed,"\mbox{\tiny bir}"]
  &&&
  &&&\GX_{r-1,r}
  \end{tikzcd}
\]
We prove that, under suitable assumptions, this factorization is related to the birational geometry of the variety $X$ (see Section \ref{sec:prun}): all the geometric quotients $\GX_{i,i+1}$ can be seen as the extremal fixed point components of $\C^*$-equivariant birational modifications $X_{i,j}$, $0\leq i<j\leq r$, of $X$, called {\em prunings}. We then have, for every $i<j$, a commutative diagram of rational maps: 

\begin{equation}\label{fig:LowerT}
\begin{tikzcd}[
  column sep={2.8em,between origins},
  row sep={2.5em,between origins}]
  \GX_{0,1}\ar[rr, dashed,"\mbox{\tiny bir.}"]&&\GX_{i,i+1}\ar[rrrr, dashed,"\mbox{\tiny bir}"]
  &&&&\GX_{j-1,j}\ar[rrrr, dashed,"\mbox{\tiny bir}"]&&&&\GX_{r-1,r}\\
  &&&&&&&&&&\\
  &&&&X_{i,j}\arrow[uurr, dashed,"\mbox{\,\,\tiny GIT quot.}",labels=below right] \arrow[uull, dashed,"\mbox{\,\,\tiny GIT quot.}",labels=below left]
  &&&&&&\\
  &&&&&&&&&&\\
  &&&&&X_{0,r}\arrow[uuuurrrrr, dashed,"\mbox{\,\,\tiny GIT quot.}",labels=below right] \arrow[uuuulllll, dashed,"\mbox{\,\,\tiny GIT quot.}",labels=below left]\arrow[uul, dashed,"\mbox{\,\,\tiny bir.}",labels=below right] 
  &&&&&  
\end{tikzcd}
\end{equation}

The technical conditions under which we prove this result are two. The first is the non existence of finite isotropy subgroups for the action  --that we call {\em equalization}-- which allows us to prove that the birational modifications $X_{i,j}$ are smooth. For instance, under this hypothesis, the variety $X_{0,r}$ is the blowup of $X$ along its sink and source.  Furthermore we assume the non existence of inner $\C^*$-invariant divisors in $X$, which implies that the birational modifications among the $\GX_{i,j}$'s and among the $X_{i,j}$'s are compositions of flips. 

Given a projective variety $X$  with a nontrivial $\C^*$-action, the  {\em normalized Chow quotient}  $\CX$ of the action is  the normalization of the closure of the subscheme of $\Chow(X)$ parametrizing the closures of the general orbits of the action (see Section \ref{sec:Chow} for details). The normalized Chow quotient can be defined in the broader setting of actions by reductive groups; in the case of torus actions it is known to have a universal property, coming from the fact that it can be seen as the normalized limit quotient of the direct system of GIT quotients of the variety in question by the action of the torus (cf. \cite[Corollary 2.6]{BHK}).

We explore here this idea, from the viewpoint of birational geometry. All the geometric quotients of the action are birationally dominated by $\CX$; in particular, we get a resolution  of the birational map $\psi:\GX_{0,1}\dashrightarrow \GX_{r-1,r}$:
$$
\xymatrix{&\CX\ar[dr]\ar[dl]&\\\GX_{0,1}\ar@{-->}[rr]^{\psi}&&\GX_{r-1,r}}
$$
It then makes sense to study the properties of this resolution and ask ourselves, for instance, whether it gives a Strong Factorization of $\psi$, that is whether the morphisms $\CX\to \GX_{0,1}$, $\CX\to \GX_{r-1,r}$ are compositions of blowups along smooth centers. What we prove is a structure theorem that shows how to obtain $\CX$ recursively out of the Chow quotients of the birational modifications $X_{i,j}$ described above, via subsequent blowups. More precisely, our main statement is the following (see $\S$\ref{sssec:quotients} for the definition of $B^\pm(Y)$):

\begin{theorem}\label{thm:main}
Let $(X,L)$ be a polarized pair, with $X$ smooth, endowed with an equalized nontrivial $\C^*$-action such that  $\codim(B^\pm(Y),X)>1$ for every inner fixed point component $Y$. Denote by $r$ be the criticality of the action, by $X_{i,j}$, $0\leq i<j\leq r$, the corresponding prunings of $X$, and by $\CX_{i,j}$ their normalized Chow quotients. Then $\CX=\CX_{0,r}$, and we have a commutative diagram:
\begin{equation}\label{fig:Chow0}
\begin{tikzcd}[
  column sep={2.7em,between origins},
  row sep={3.2em,between origins}]
&&&&&&\CX_{0,r}\arrow[rd,"\prur"] \arrow[dl,"\prul",labels=above left] &&&&&&\\
&&&&&\CX_{0,r-1}\arrow[rd,"\prur"] \arrow[dl,"\prul",labels=above left]&&\CX_{1,r}\arrow[rd,"\prur"] \arrow[dl,"\prul",labels=above left]&&&&&\\
&&&&\CX_{0,r-2} \arrow[dl,"\prul",labels=above left]\arrow[rd,"\prur"]&&\CX_{1,r-1}\arrow[rd,"\prur"] \arrow[dl,"\prul",labels=above left]&&\CX_{2,r}\arrow[dl,"\prul",labels=above left]\arrow[rd,"\prur"]&&&&\\
&&& \arrow[ld,dotted,no head]&& \arrow[ld,dotted,no head]&&\arrow[rd,dotted,no head]&&\arrow[rd,dotted,no head]&&&\\
&&\CX_{0,2}\arrow[rd,"\prur"] \arrow[dl,"\prul",labels=above left]&&\CX_{1,3} \arrow[dl,"\prul",labels=above left]&&\dots&&\CX_{r-3,r-1}\arrow[rd,"\prur"] &&\CX_{r-2,r}\arrow[rd,"\prur"] \arrow[dl,"\prul",labels=above left]&&\\
&\GX_{0,1}\arrow[rd]&&\GX_{1,2}\arrow[ld]\arrow[rd]&&\dots&&\dots&&\GX_{r-2,r-1}\arrow[ld]\arrow[rd]&&\GX_{r-1,r}\arrow[ld]&\\
&&\GX_{1,1}&&\GX_{2,2}&&\dots&&\GX_{r-2,r-2}&&\GX_{r-1,r-1}&&
\end{tikzcd}
\end{equation}
where all the arrows denoted by $\prul,\prur$ are blowups and every rhombus in the diagram is a normalized Cartesian square.
\end{theorem}

In particular, the whole diagram is constructed upon its first two  layers (that contain the GIT quotients of $X$) by successive normalized Cartesian squares. In particular, this result implies that every chain of closures of orbits linking the sink and the source of $X$ is parametrized by an element of $\CX$ (Corollary \ref{cor:CXcontainsall}). Moreover, our proof of the statement will provide an explicit description of the centers of all the blowup maps in the diagram (Remark \ref{rem:center}). 

We also present an example that shows that $\CX$ can be singular (Example \ref{ex:BRUS}), and show that it is smooth if $X$ is convex, i.e. if $\HH^1(\P^1,\mu^*TX)=0$ for every morphism $\mu:\P^1\to X$. More precisely, we prove the following:
\begin{theorem}\label{thm:convex}
Under the assumptions of Theorem \ref{thm:main}, assume moreover that $X$ is convex. Then the varieties $\CX_{i,j}$ with $i=0$ or $j=r$ are smooth, and the maps 
\[\xymatrix{\GX_{0,1}&\CX_{0,2}\ar[l]_(0.44)s&\dots\ar[l]_(0.42)s&\CX\ar[l]_(0.42)s\ar[r]^(0.44)d&\dots\ar[r]^(0.40)d&\CX_{r-2,r}
\ar[r]^(0.46)d&\GX_{r-1,r}}\]
are smooth blowups, that is, we have a strong factorization of $\psi$.
\end{theorem}
The smoothness of the intermediate Chow quotients in the above factorization follows from using Theorem \ref{thm:main} together with some ideas from \cite{FP,MMW}, which allows us to prove the unobstructedness of deformations of the cycles parametrized by the elements of $\CX$ under the convexity assumption (see also Remark \ref{rem:convorbits}). The smoothness of the centers of the blowups follows then by a statement due to Andreatta and Wi\'sniewski \cite[Corollary~4.11]{AWDuke}.

% prove that $\CX$ is smooth if $X$ is a convex variety. Moreover, we prove that under this assumption all the varieties $\CX_{i,j}$ with $i=0$ or $j=r$ are smooth, and the maps 
%\[\xymatrix{\GX_{0,1}&\CX_{0,2}\ar[l]_s&\dots\ar[l]_s&\CX\ar[l]_s\ar[r]^d&\dots\ar[r]^d&\CX_{r-2,r}
%\ar[r]^d&\CX_{r-1,r}}\]
%are smooth blowups (see Proposition \ref{prop:smoothChows}, see also Remark \ref{rem:resolution}). In other words, we get a strong factorization of the birational map $\psi$.
%%$s:\CX_{0,j+1}\to\CX_{0,j}$, $d:\CX_{i+1,r}\to \CX_{i,r}$ 

\subsection*{Outline} The structure of the paper is the following. The first two sections contain some preliminary material on $\C^*$-actions and prunings, in which we essentially recall some useful results and notations presented in \cite{WORS3, BRUS}. Note that the main result of Section \ref{sec:prun} (Proposition \ref{prop:WORS3}) is presented under slightly different hypotheses than in the original source \cite{WORS3}. Section \ref{sec:Chow} is devoted to the Chow quotient of a $\C^*$-action; a key result that we prove is that under the equalization hypothesis, the Chow quotient is the normalization of a subscheme of the Hilbert scheme of $X$ (Remark \ref{rem:Hilb}). Furthermore, we show how to effectively write $\CX$ as the normalization of a certain subvariety of a Grassmannian (Corollary \ref{cor:chowgrass}). The proofs of Theorems \ref{thm:main} and \ref{thm:convex} are presented in Sections \ref{sec:main} and \ref{sec:smooth}, respectively; the latter contains also a toric example whose Chow quotient is not smooth. We conclude the paper with Section \ref{sec:final}, in which we discuss the assumption on the codimension of $B^\pm(Y)$ for every inner fixed point component $Y\in\cY^\circ$, and show that Theorem \ref{thm:main} works under milder hypotheses. As an example, we consider the Chow quotient of the diagonal $\C^*$-action of $(\P^1)^n$, which is known to be isomorphic to the Losev--Manin space (see \cite{DRS} and references therein); our techniques allow us to describe the contractions of this variety appearing in Figure \ref{fig:Chow0}.

%%%%%%%%%%%%%%%%%%%%%%%%%%%%%%%%%%%%%%%%%%%%%%%%%%%%%%%%%%%%%%%%%%%%%%%%%%%%%%%
%%%%%%%%%%%%%%%%%%%%%%%%%%%%%%%%%%%%%%%%%%%%%%%%%%%%%%%%%%%%%%%%%%%%%%%%%%%%%%%

\section{Preliminaries on $\C^*$-actions and GIT quotients}\label{sec:prelim}

Along this section $X$ will denote a normal complex projective variety, with a nontrivial $\C^*$-action. 

%%%%%%%%%%%%%%%%%%%%%%%%%%%%%%%%%%%%%%%%%%%%%%%%%%%%%%%%%%%

\subsection{Fixed point components}\label{sssec:fxd} We denote by $X^{\C^*}$  the fixed locus of the action, and by $\cY$ the set of irreducible fixed point components: 
$$X^{\C^*}=\bigsqcup_{Y\in \cY}Y.$$ 
 The {\em sink} and the {\em source} of the action are the fixed point components containing, respectively, the limiting points of  a general $x \in X$:
$$
\lim_{t\to 0}t^{-1}x, \quad \lim_{t\to 0}tx.
$$
 If $X$ is smooth and projective, every $Y\in\cY$ is smooth (cf. \cite{IVERSEN}).

%%%%%%%%%%%%%%%%%%%%%%%%%%%%%%%%%%%%%%%%%%%%%%%%%%%%%%%%%%%

\subsection{Linearizations and weight maps}\label{sssec:lin} Given  a line bundle $L\in \Pic(X)$, one may find a linearization of the $\C^*$-action on it (cf. \cite{KKLV}), so that for every $Y\in \cY$, $\C^*$ acts on $L_{|Y}$ by multiplication with a character $m\in \Mo(\C^*)=\Hom(\C^*,\C^*)$, called {\em weight of the linearization on $Y$}. Fixing an isomorphism $\Mo(\C^*)\simeq \Z$, a linearization defines a map $\mu_L:\cY\to\Z$, such that 
$$
\mu_{kL}(Y)=k\mu_L(Y),\quad \mu_{L+L'}(Y)=\mu_{L}(Y)+\mu_{L'}(Y),
$$
for every $L,L'\in\Pic(X)$, $k\in \Z$, $Y\in \cY$.

%%%%%%%%%%%%%%%%%%%%%%%%%%%%%%%%%%%%%%%%%%%%%%%%%%%%%%%%%%%

\subsection{Actions on polarized pairs}\label{sssec:polar} A {\em $\C^*$-action on the polarized pair $(X,L)$}
is a $\C^*$-action on a projective variety $X$, together with the weight map $\mu_L$ determined by an ample line bundle $L$.  Since two linearizations differ by a character of $\C^*$, for $L \in \Pic(X)$ there exists a unique linearization (called {\em normalized})  whose smaller weight is equal to zero.
In this case we denote by 
$$
0=a_0<\dots<a_r,$$ 
the weights $\mu_L(Y)$, $Y\in\cY$,  
and set: 
$$
Y_i:=\bigcup_{\mu_L(Y)=a_i}Y.
$$ 
The values $a_i$ will be called the {\em critical values} and the number $r$ will be called the {\em criticality} of the action.
The minimum and maximum of the critical values are achieved only at the  sink and the  source of the action, which are then equal to $Y_0$ and $Y_r$, respectively (see \cite[Remark 2.12]{WORS1}). The value $\delta=a_r$ is called {\em bandwidth} of the $\C^*$-action on $(X,L)$. The fixed point components different from $Y_0$, $Y_r$ are called {\em inner}, and their set is denoted by $$\cY^\circ :=\cY\setminus\{Y_0,Y_r\}.$$

%%%%%%%%%%%%%%%%%%%%%%%%%%%%%%%%%%%%%%%%%%%%%%%%%%%%%%%%%%%

\subsection{Geometric quotients}\label{sssec:quotients} It is a well known fact in Mumford's Geometric Invariant Theory that a $\C^*$-action on a polarized variety $(X,L)$  allows to define projective quotients of $X$, as follows. Let 
\[A:=\bigoplus_{m\geq 0}\HH^0(X,mL)\] 
be the {\em graded $\C$-algebra of sections of $L$}, and set, for every $\tau \in [0,\delta]\cap \Q$:
 \[A(\tau):=\bigoplus_{\substack{m\geq 0\\m\tau\in\Z}}\HH^0(X,mL)_{m\tau},\qquad \GX(\tau):=\Proj(A(\tau)).
\]
We will call the  $\GX(\tau)$'s the {\em GIT-quotients of the $\C^*$-action on $(X,L)$}.

The above construction has been extended by Bia{\l}ynicki-Birula and Swi\c{e}cicka to the construction of other (in general non projective) quotients of $X$, described in terms of semi-sections and sections of the $\C^*$-action (cf. \cite{BBS}; see also \cite{WORS3}). Their construction allows us to describe easily the open subsets $X^{\ss}(\tau)\subset X$ of semistable points on which the rational maps
$$
X\dashrightarrow \GX(\tau)
$$
are defined. 
In order to do so we need the language of {\em BB-cells}, provided by the {\em Bia{\l}ynicki-Birula decomposition} of $X$ (see \cite{BB,CARRELL}): given a fixed point component $Y\in \cY$, we denote by 
\begin{equation}\label{eq:BBcells}
X^\pm(Y):=\displaystyle\{x\in X|\,\, \lim_{t^{\pm 1}\to 0} tx\in Y\},\qquad 
B^\pm(Y):=\overline{X^\pm(Y)},
\end{equation}
the {\it BB-cells} of the action and their closures. The equalities:
\[X=\bigsqcup_{Y\in \cY}X^+(Y)=\bigsqcup_{Y\in \cY}X^-(Y)
\]
are called the {\em $+$} and {\em $-$ BB-decomposition} of $X$. 

To describe $X^{\ss}(\tau)$ for a $\C^*$-action on a polarized pair $(X,L)$, let us define
$$B^\pm_i:=\bigcup_{\mu_L(Y)=a_i}B^\pm(Y)$$
and, for notational reasons, $B^\pm_i=\emptyset$ for $i\in \Z\setminus\{0,\dots,r\}$. Then,
following \cite{BBS}, we may write:
\[
X^{\ss}(\tau)=X\setminus \big(\bigcup_{a_i<\tau}B^+_i\cup \bigcup_{a_i>\tau}B^{-}_i\big).
\]
In particular, $X^{\ss}(\tau)$ is the same for any  $\tau\in (a_i,a_{i+1})\cap \Q$, and so we set:
\begin{itemize}[itemsep=5pt,topsep=5pt]
\item $\GX_{i,i}:=\GX(\tau)$ for $\tau=a_i$; 
\item $\GX_{i,i+1}:=\GX(\tau)$ for $\tau\in(a_i,a_{i+1})\cap \Q$.
\end{itemize}
Note that $\GX_{i,i+1}$ is a geometric quotient of $X$ for every $i$, whereas $\GX_{i,i}$ is only semigeometric.  
Moreover, by construction, $$\GX_{0,0}=Y_0,\quad \GX_{r,r}=Y_r.$$ 
Furthermore, since the intersection of all the open sets $X^{\ss}(i,i+1)$ is nonempty, the varieties $\GX(i,i+1)$ are birationally equivalent. The natural birational maps among them fit in the following commutative diagram, whose diagonal arrows are contractions:
\begin{equation}
\begin{tikzcd}[
  column sep={3.4em,between origins},
  row sep={3.5em,between origins},
]
&\GX_{0,1} \arrow[rr,dashed] \arrow[rd] \arrow[dl]&&\GX_{1,2} \arrow[dl]&\dots\ \ &\GX_{r-2,r-1} \arrow[rr,dashed] \arrow[rd] &&\GX_{r-1,r}\arrow[rd] \arrow[dl]&\\
\GX_{0,0}&&\GX_{1,1}&\dots&&\dots\ \ &\GX_{r-1,r-1}&&\GX_{r,r}
\end{tikzcd}
\label{eq:GITquot}
\end{equation}
The horizontal maps are Atiyah flips; this is a known fact which goes back to the work of Thaddeus, Reid and W{\l}odarczyk (\cite{Wlodarczyk2009,Thaddeus1996,ReidFlip}).

\begin{example}\label{ex:toric}
An obvious source of examples of $\C^*$-actions is the category of toric varieties. An algebraic variety $X$ is called {\em toric} if it admits an action of an algebraic torus $\mathbb{T}=(\mathbb{C}^*)^n$ with an open orbit. As usual, we denote by $M=\Mo(\TT):=\Hom_{\alg}(\mathbb{T},\C^*)$ and $N:=\Hom_{\alg}(\C^*,\mathbb{T})$ the lattices of characters and 1-parameter subgroups of $\mathbb{T}$. If $X$ is normal, it is completely described by a regular fan $\Sigma$ in $N\otimes\mathbb{R}$.  
If $X$ is projective and we choose an ample line bundle $L$ on $X$, together with a linearization of the action of $\mathbb{T}$, then the weights of the torus action on the space of sections of $L$ are lattice points in a lattice polytope $\Delta\subset M\otimes\mathbb{R}$; the lattice polytope determines the fan $\Sigma$, hence the variety $X$, so it makes sense to write $X=X(\Delta)$. The vertices of $\Delta$ represent the weights of the action of $\mathbb{T}$ on fibers of $L$ over fixed points of the action on $X$. 
  
Actions of $\C^*$ on a projective toric variety $X=X(\Delta)$ as above can be  obtained by considering 1-parameter subgroups $\nu\in N$, which are determined by projections $\nu: M\otimes\mathbb{R}\rightarrow\mathbb{Z}\otimes\mathbb{R}$. The critical values of the corresponding $\C^*$-action (together with the induced linearization on $L$) are the images of vertices of $\Delta$ under the projection $\nu$. The associated GIT quotients of the action on $(X,L)$ are then toric varieties, with respect to the quotient torus $T/\nu\otimes\mathbb{C}^*$. They can be described as the toric varieties determined by the polytopes $\nu^{-1}(a)\cap\Delta$, $a\in\mathbb{Q}\cap\nu(\Delta)$, obtained by slicing $\Delta$ with hyperplanes of the form $\nu^{-1}(a)$. We refer to \cite{KSZ} for details.
\end{example}

%%%%%%%%%%%%%%%%%%%%%%%%%%%%%%%%%%%%%%%%%%%%%%%%%%%%%%%%%%%

\subsection{The smooth case}\label{sssec:BBsmooth}
If $X$ is smooth, the normal bundle in $X$ of a fixed point component $Y$  splits into two subbundles, on which $\C^*$ acts with positive and negative weights, respectively:
\begin{equation}
\cN_{Y|X}\simeq \cN^+(Y)\oplus \cN^-(Y).
\label{eq:normal+-}
\end{equation}
We set
\begin{equation}\label{eq:Vpm}
\nu^{\pm}(Y):=\rk (\cN^{\pm}(Y)), \qquad 
V^\pm(Y):=\P(\cN^{\pm}(Y)^\vee).
\end{equation}

%%%%%%%%%%%%%%%%%%%%%%%%%%%%%%%%%%%%%%%%%%%%%%%%%%%%%%%%%%%

\subsection{Equalized actions}\label{sssec:equal}

We say that a $\C^*$-action  is {\em  equalized} at $Y\in\cY$ if for every $x\in \big(X^{-}(Y)\cup X^{+}(Y)\big)\setminus Y$ the isotropy group of the action at $x$ is trivial. In the smooth case,  as shown in \cite[Lemma 2.1]{WORS3}, one may ask, equivalently, that for
every  fixed point component 
$Y$, $\C^*$ acts on $\cN^+(Y)$ with all the weights equal to $+1$ and on $\cN^-(Y)$ with all the weights equal to $-1$ (\cite[Definition 1]{RW}). 

\begin{remark}\label{rem:normalquot}
If $X$ is smooth and the $\C^*$-action is equalized, the geometric quotients $\GX_{i,i+1}$ are smooth, as shown in  \cite[Lemma 2.15]{WORS3}. Moreover, under these hypotheses, if $Y$ is a fixed point component of weight $\mu_L(Y)=a_i$, then  $V^+(Y)$ (respectively $V^-(Y)$) can be embedded into $\GX_{i-1,i}$  (resp. $\GX_{i,i+1}$) as the parameter space of orbits with source (resp. sink) at $Y$.
\end{remark}

The equalization hypothesis implies that the closure of any $1$-dim\-ens\-ional orbit is a smooth rational curve, whose $L$-degree may be computed in terms of the weights at its extremal points. 
The following statement (\cite[Lemma 2.2]{WORS3}) follows from the AM vs. FM formula presented in \cite[Corollary 3.2]{RW}: 

\begin{lemma}[AM vs. FM]\label{lem:AMFM}
Let $(X,L)$ be a polarized pair, with $X$ smooth, supporting an equalized $\C^*$-action, and let $C$ be the closure of a $1$-dimensional orbit, whose sink and source are denoted by $x_-$ and $x_+$. Then $C$ is a smooth rational curve of $L$-degree  equal to $\mu_L(x_+)-\mu_L(x_-)$.
\end{lemma} 

\subsection{B-type actions and bordisms}\label{sssec:btype} 
Following \cite[Section 3]{WORS1}, we call  {\em B-type action} a $\C^*$-action whose extremal fixed point components $Y_0,Y_r$ have codimension one. 

If $X$ is smooth and projective, by the Bia{\l}ynicki-Birula decomposition, the restriction maps $\Pic(X)\to \Pic(Y_j)$, $j=0,r$ are surjective (cf. \cite[Theorem~3]{CS79}). We will say that the $\C^*$-action is a {\em bordism} (cf. \cite[Definition~3.8]{WORS1}) if the restriction maps $\imath_j^*:\Pic(X)\to \Pic(Y_j)$, $j=0,r$, fit into two short exact sequences:
$$
\begin{array}{c}
\shse{\Z[Y_r]}{\Pic(X)}{\Pic(Y_0)}
\\[2pt]
\shse{\Z[Y_0]}{\Pic(X)}{\Pic(Y_r)}
\end{array}
$$

The following equivalence, which is a consequence of \cite[Theorem~3]{CS79} has been proved in \cite[Corollary~3.7]{WORS1}.
\begin{lemma}\label{lem:bordism} Let $X$ be a smooth projective variety admitting a $\C^*$-action of B-type. The action is a bordism if and only if $\nu^\pm(Y)=\codim(B^\pm(Y))\geq 2$ for every inner fixed point component $Y\in\cY^\circ$. 
\end{lemma}

%%%%%%%%%%%%%%%%%%%%%%%%%%%%%%%%%%%%%%%%%%%%%%%%%%%%%%%%%%%

\subsection{A digression on Grassmannians and $\C^*$-actions}\label{ssec:grass}
 
In this section we discuss an example that we will use later on. 

Let $V$ be a finite dimensional complex vector space with a nontrivial linear action of $\C^*$ with weights $a_0<\dots<a_s\in\Z$, inducing a decomposition:
$$
V=\bigoplus_{k=0}^s V_k,\qquad V_k=\{v\in V|\,\, t(v)=t^{a_k}v\,\,\mbox{for all }t\in\C^*\}.
$$
The induced action on the (Grothendieck) projectivization $\P(V)$ of $V$ has exactly $s+1$ fixed-point components $\P(V_k)\subset \P(V)$. The sink and the source of the action are $\P(V_0),\P(V_s)$, and there exists a linearization of the action on the tautological line bundle $\cO(1)$ such that:
$$
\mu_{\cO(1)}(\P(V_k))=a_k.
$$ 
Let us denote by $\G(s,\P(V))$ the Grassmannian of $s$-dimensional projective subspaces of $\P(V)$ and consider the induced $\C^*$-action. One of its fixed point components is the variety parametrizing subspaces meeting each $\P(V_k)$ in one point, which is isomorphic to the Segre variety $\prod_{k=0}^s\P(V_k)$; we will denote it by  
\begin{equation}\label{eq:Sigma}\Sigma(V):=\{\P(W)\subset \P(V)|\,\, \dim(\P(W)\cap \P(V_k))=0,\,\, \forall k\}\subset \G(s,\P(V))\end{equation}
Now let us consider two given indices $k_-,k_+$ with $0\leq k_-<k_+\leq s$, and set:
$$
V_{k_-,k_+}:=\bigoplus_{k=k_-}^{k_+} V_k.
$$
We have induced $\C^*$-actions on $V_{k_-,k_+}$ (with weights $a_{k_-}<\dots<a_{k_+}$), $\P(V_{k_-,k_+})$ and the Grassmannian $\G(k_+-k_-,\P(V_{k_-,k_+}))$. One of the fixed-point components of the action on $\G(k_+-k_-,\P(V_{k_-,k_+}))$ is $\Sigma(V_{k_-,k_+})$, which is isomorphic to $\prod_{k=k_-}^{k_+}\P(V_k)$; note also that the restriction of the Pl\"ucker embedding of the Grassmannian  $\G(k_+-k_-,\P(V_{k_-,k_+}))$ to $\Sigma(V_{k_-,k_+})$ is the Segre embedding.

The linear projection $\varphi_{k_-,k_+}:\P(V)\dashrightarrow \P(V_{k_-,k_+})$ is $\C^*$-equivariant and induces a surjective morphism 
$$
\pi_{k_-,k_+}:\Sigma(V)\lra \Sigma(V_{k_-,k_+}),
$$ 
which corresponds to the natural projection $$\prod_{k=0}^s\P(V_k)\lra\prod_{k={k_-}}^{k_+}\P(V_k),\qquad (P_0,\dots,P_s)\mapsto (P_{k_-},\dots,P_{k_+}).$$

%%%%%%%%%%%%%%%%%%%%%%%%%%%%%%%%%%%%%%%%%%%%%%%%%%%%%%%%%%%%%%%%%%%%%%%%%%%%%%%%%%%%%%%%%%%%%%%%
%%%%%%%%%%%%%%%%%%%%%%%%%%%%%%%%%%%%%%%%%%%%%%%%%%%%%%%%%%%%%%%%%%%%%%%%%%%%%%%%%%%%%%%%%%%%%%%%

\section{Prunings}\label{sec:prun}

Following \cite{WORS3}, the GIT quotients of $X$ can be seen as the extremal fixed point components of $\C^*$-equivariant birational modifications of $X$, that are referred to as {\em prunings} of $X$ in \cite{BRUS}.

\begin{definition}\label{def:pruning}
Let $(X,L)$ be a polarized pair, and let $\tau_-<\tau_+$ be two rational numbers in $[0,\delta]$. We define the {\em pruning} of $(X,L)$ with respect to $\tau_-,\tau_+$ as: 
	$$
	X(\tau_-,\tau_+):=\Proj A(\tau_-,\tau_+),\ \mbox{where}\ A(\tau_-,\tau_+)=\bigoplus_{m\in d\Z_{\geq 0}}\bigoplus_{k \in[m\tau_-, m\tau_+]}\HH^0(X,mL)_k,
	$$	
and $d$ is the minimum positive integer such that $d\tau_\pm\in \Z$.
\end{definition}
	
\begin{remark}\label{rem:pruning} Note that $A(\tau_-,\tau_+)$ is finitely generated by \cite[Lemma 3.2]{BRUS}, guaranteeing the good definition of $X(\tau_-,\tau_+)$.  
\end{remark}

\begin{example}\label{ex:toric1}
Assume that $X=X(\Delta)$ is a projective toric variety, polarized with an ample line bundle $L$, associated with a lattice polytope $\Delta\subset M\otimes \R$,  with a $\C^*$-action determined by a $1$-parameter subgroup $\nu$. Then, with the notation of Example \ref{ex:toric}, the prunings of $X$ are the  toric varieties associated with polytopes of the form $\nu^{-1}[a,b]\cap\Delta$ where $a<b$ are rationals contained in $\nu(\Delta)$. 
\end{example}

In the case in which $X$ is smooth and the action is equalized, the varieties $X(\tau_-,\tau_+)$ can be easily written as birational modifications of $X$, as follows. Let us denote by $\beta:X^\flat\lra X$ the blowup of $X$ along the sink and the source $Y_0$, $Y_r$. The $\C^*$-action on $X$ extends via $\beta$ to an action on $X^\flat$, whose sink and source are the exceptional divisors over $Y_0,Y_r$, that we denote by $Y^\flat_0,Y^\flat_r$. Let us define
\[
L(\tau_-,\tau_+):=\beta^*L-\tau_-Y^\flat_0-(\delta-\tau_+)Y^\flat_r.
\]
Then, following \cite[Lemma 2.5]{WORS3}, we can write:
\[
A(\tau_-,\tau_+)=\bigoplus_{m\in d\Z_{\geq 0}}\HH^0(X^\flat,mL(\tau_-,\tau_+)).\]

\begin{remark}\label{rem:Xflat}
In particular, for $\tau_-$, $\tau_+$ sufficiently close to $0$, $\delta$, the line bundle $L(\tau_-,\tau_+)$ is ample on $X^\flat$, so in this case  $X(\tau_-,\tau_+)=X^\flat$. 
\end{remark}

Furthermore, we can describe the base loci of the line bundles  $mL(\tau_-,\tau_+)$ for $m$ large enough and divisible (cf. \cite[Corollary~2.7]{WORS3}):
\begin{lemma}\label{lem:baselocus}
Let $(X,L)$ be a polarized pair, with $X$ smooth,  with an equalized $\C^*$-action. With the above notation, the base locus of $mL(\tau_-,\tau_+)$ is equal to the strict transform in $X^\flat$ of 
\[
\bigcup_{a_i< \tau_-}B^+_i\cup\bigcup_{a_j> \tau_+}B^-_j.
\]
\end{lemma}
In particular, $\Bs(mL(\tau_-,\tau_+))$ depends only on:
$$
i(\tau_-):=\min\{i|\,\,a_i\geq\tau_-\}-1,\quad j(\tau_+):=\max\{j|\,\,a_j\leq\tau_+\}+1,
$$
so we may define:
\[
\cB_{i(\tau_-)j(\tau_+)}:=\Bs(mL(\tau_-,\tau_+))=\bigcup_{a_i< \tau_-}B^+_i\cup\bigcup_{a_j> \tau_+}B^-_j,
\]
and assert also that $\cB_{i(\tau_-)j(\tau_+)}$ is equal to the {\em stable base locus of }$mL(\tau_-,\tau_+)$, that we denote by $\B(mL(\tau_-,\tau_+))$.

\begin{notation}\label{not:cones}
Given line bundles $L_1,L_2,\ldots, L_i$ on a projective variety $M$, we will denote by $$\langle L_1,L_2,\dots, L_i\rangle$$ the convex cone of nonnegative real linear combinations of the numerical classes of these line bundles in the space $N^1(M)$ (the real vector space generated by Cartier divisors in $M$ modulo numerical equivalence). 
\end{notation}

We will now consider a $\C^*$-action on a polarized pair $(X,L)$ with $X$ smooth, for which we will use the notation of Section \ref{sec:prelim}. Our goal will be to study the intersection of the movable cone of $X^\flat$ with $\langle \beta^*L,Y^\flat_0,Y^\flat_r\rangle$. For simplicity we will assume that the $\C^*$-action is a bordism, which, by our discussion above, is equivalent to say that $\B(mL(\tau_-,\tau_+))$ contains no codimension one subvarieties, for every $\tau_-,\tau_+$.

\begin{remark}\label{rem:picz}
Note that we are not assuming any particular value of the Picard number of $X$. Our arguments we will be similar to the ones considered in \cite{WORS3}, where we assumed that $\Pic(X)\simeq \Z$, but considered also the non-bordism case. 
\end{remark}
 
\begin{lemma}\label{lem:movcone} Let $(X,L)$ be a polarized pair, with $X$ smooth, endowed with an equalized $\C^*$-action such that the induced action on
$X^\flat$ is a bordism. Then
\[\ol{\Mov(X^\flat)}\cap \langle \beta^* L,Y^\flat_0,Y^\flat_r\rangle= \langle \beta^*L,L(0,0),L(\delta,\delta) \rangle.\]
\end{lemma}

\begin{proof}
The proof follows verbatim \cite[Proposition~4.7]{WORS3}. Note that the assumption on the Picard number of $X$ in the quoted reference
was used only to guarantee that  $\codim(B^\pm(Y),X) \ge 2$  for every $Y\in \cY^\circ$.
\end{proof}
Combining  Lemmas \ref{lem:movcone} and \ref{lem:baselocus}, we get the {\em stable base locus decomposition} of the cone $\ol{\Mov(X^\flat)}\cap \langle \beta^* L,Y^\flat_0,Y^\flat_r\rangle$, that we have represented in Figure \ref{fig:movbordism}; 
the stable base loci (open) chambers are denoted by $N_{i,j}$, and defined as:
\[
N_{i,j}=\begin{cases}
\Int\langle L(a_i,a_{j-1}),L(a_{i+1},a_{j-1}),L(a_i,a_{j}),L(a_{i+1},a_{j})\rangle & \mbox{ if }j-i\geq 2,\\[3pt]
\Int\langle L(a_i,a_{i}),L(a_i,a_{i+1}),L(a_{i+1},a_{i+1})\rangle & \mbox{ if }j=i+1.
\end{cases}
\] 
for $0 \le i < j \le r$.
\begin{figure}[h!!]
\begin{tikzpicture}[scale=0.9]
\draw (-6,0) -- (-4,0) -- (-5,-1) -- (-6,0);
\draw (6,0) -- (4,0) -- (5,-1) -- (6,0);
\draw [thick] (-6,0) -- (6,0) -- (0,-6) -- (-6,0);
\draw (-4,0) -- (1,-5);
\draw (-4,-2) -- (-2,0) -- (2,-4);
\draw (-3,-3) -- (0,0) -- (3,-3);
\draw (-2,-4) -- (2,0) -- (4,-2);
\draw (-1,-5) -- (4,0); 
\draw [fill=blue!10] (0,-6)--(1,-5)--(0,-4)--(-1,-5)--(0,-6);
\node [below] at (-4,-4) {$\Mov(X^\flat)$}; \node at (0,-5) {$N_{0,r}$};
\node at (-1,-4) {$N_{0,r-1}$};\node at (1,-4) {$N_{1,r}$};
\node at (-2,-3) {$N_{0,r-2}$};\node at (0,-3) {$N_{1,r-1}$};\node at (2,-3) {$N_{2,r}$};
\node at (-3,-2) {$\ddots$};\node at (-1,-2) {$\ddots$};\node at (1,-2) {$\iddots$};\node at (3,-2) {$\iddots$};
\node at (-4,-1) {$N_{0,2}$};\node at (-2,-1) {$N_{1,3}$};\node at (0,-1) {$\ldots$};\node at (2,-1) {$\ldots$};\node at (4,-1) {$N_{r-2,r}$};
\node [below] at (-5,0.00) {$N_{0,1}$};\node [below] at (-3,0.00) {$N_{1,2}$};\node [below] at (-1,-0.18) {$\ldots$};\node [below] at (1,-0.18) {$\ldots$};\node [below] at (3,-0.18) {$\ldots$} ;\node [below] at (5,0.05) {$N_{r-1,r}$};
\fill[black!90!white] (-6,0) circle (0.6mm); \node[anchor=south]  at (-6,0) {$L(0,0)$\qquad};
\fill[black!90!white] (6,0) circle (0.6mm); \node[anchor=south] at (6,0) {$\qquad\,\,L(\delta,\delta)$};
\fill[black!90!white] (0,-6) circle (0.6mm); \node[anchor=north] at (0,-6) {$\beta^*L=L(0,\delta)$};
\end{tikzpicture}
\caption{SBL decomposition of $\ol{\Mov(X^\flat)}\cap \langle \beta^* L,Y^\flat_0,Y^\flat_r\rangle$.
\label{fig:movbordism}
} \end{figure}
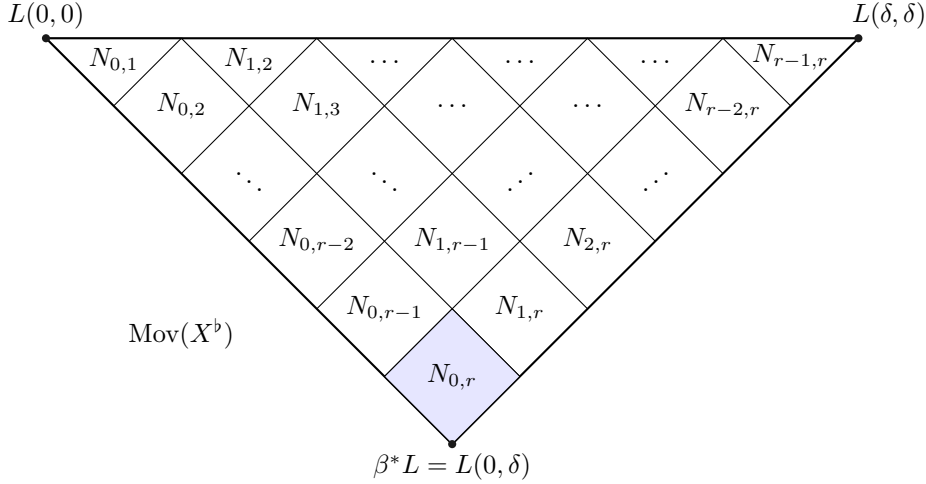

Finally, following the lines of argumentation of \cite[Proposition 4.11]{WORS3}, we can prove  that the stable base locus decomposition of the cone $\ol{\Mov(X^\flat)}\cap \langle \beta^* L,Y^\flat_0,Y^\flat_r\rangle$  coincides with its Mori chamber decomposition. 

\begin{proposition}\label{prop:WORS3}
Let $(X,L)$ be a polarized pair, with $X$ smooth,  with an equalized $\C^*$-action such that  $\codim(B^\pm(Y),X)>1$ for every inner fixed point component $Y\in\cY^\circ$.
Then for every $0 \le i \le j \le r$ there exist a smooth projective variety $X_{i,j}$ and a small $\Q$-factorial modification $\varphi_{i,j}:X^{\flat}\dashrightarrow X_{i,j}$ such that
\begin{itemize}
\item $X_{i,j}=X(\tau_-,\tau_+)$ whenever $\tau_-\in(a_i,a_{i+1})$, $\tau_+\in (a_{j-1},a_j)$; 
\item $\Nef(X_{i,j})\cap \langle \beta^* L,Y^\flat_0,Y^\flat_r\rangle=\ol{N_{i,j}}$.
\end{itemize}
\end{proposition} 

\begin{proof}
We start from $X_{0,r}:=X^\flat$, that satisfies $\Nef(X_{0,r})=\ol{N_{0,r}}$. Then for every pair of indices $i,j$ such that $0 \le i \le j \le r$, we construct a small $\Q$-factorial modification $\varphi_{i,j}:X_{0,r}\dashrightarrow X_{i,j}$ (called {\em pruning map}), such that

\begin{itemize}[topsep=10pt]
\item the variety $X_{i,j}$ is smooth, projective, with a $\C^*$-action such that $\varphi_{i,j}$ is $\C^*$-equivariant;
\item the set of inner fixed-point components of the action on  $X_{i,j}$ is:  
$$
\{\varphi_{i,j}(Y)|\,\,Y\in\cY^\circ,\,\, a_i<\mu_L(Y)<a_j\},
$$
and $\varphi_{i,j}$ is an isomorphism on an open set containing these components.
\end{itemize}
These modifications are constructed recursively as follows: assuming we have already constructed the variety $X_{i,j}$, and the map $\varphi_{i,j}$ (with $i<j-1$), then the varieties $X_{i,j-1},X_{i+1,j}$ are defined by means of small $\Q$-factorial modifications $\prul: X_{i,j} \dashrightarrow X_{i,j-1}$, $\prur: X_{i,j} \dashrightarrow X_{i+1,j}$  constructed as in \cite[Theorem 3.1]{WORS3}.

Namely $X_{i,j-1}$ (resp. $X_{i+1,j}$) is constructed as the smooth blowup $b_\prul:X^\prul_{i,j} \to X_{i,j}$ (resp. $b_\prur:X^\prur_{i,j} \to X_{i,j}$) along the strict transform into $X_{i,j}$ of $B^-_{j-1}$ (resp. $B^+_{i+1}$), followed by a smooth blowdown with the same exceptional locus; this blowdown contracts the class of the closure of a $1$-dimensional $\C^*$-orbit in $\Exc(b_\prul)$ (resp. $\Exc(b_\prur)$). 
\begin{equation}\label{eq:blowups}\begin{gathered}
\xymatrix{&X^\prul_{i,j}\ar[ld]\ar[rd]^{b_\prul}&&X^\prur_{i,j}\ar[ld]_{b_\prur}\ar[rd]&\\
X_{i,j-1}&&X_{i,j}\ar@{-->}[ll]^{\prul}\ar@{-->}[rr]_{\prur}&&X_{i+1,j}}
\end{gathered}
\end{equation}

In particular, the indeterminacy loci of $\prul,\prur$ are the strict transforms of $B^-_{j-1}$, $B^+_{i+1}$ into $X_{i,j}$, respectively, and the criticality of the $\C^*$-action gets reduced by one via these maps.  
We then set $$\varphi_{i+1,j}:=\prur\circ\varphi_{i,j},\qquad\varphi_{i,j-1}:=\prul\circ\varphi_{i,j}.$$

Finally, arguing as in the proof of \cite[Proposition 4.11]{WORS3} we can show that, identifying the space $\NU(X_{i,j})$ with $\NU(X^\flat)$ we have $\Nef(X_{i,j})\cap \langle \beta^* L,Y^\flat_0,Y^\flat_r\rangle = \ol{N_{i,j}}$. In other words, 
$$
X_{i,j}=X(\tau_-,\tau_+), \mbox{ whenever } i(\tau_-)=i, \,\,j(\tau_+)=j.
$$
This completes the proof.
\end{proof}

\begin{remark}\label{rem:WORS3}
Note that by construction, the maps $\prul,\prur$ commute, so that every map $\varphi_{i,j}$ can be written as a composition $$\varphi_{i,j}=\prul^{i}\circ\prur^{r-j}.$$
We will call $\prul,\prur$ {\em elementary pruning maps}.
\end{remark}

Let us finish this section relating  extremal fixed point components of the varieties $X^s_{i,j}$ and $X^d_{i,j}$  
appearing in the above proof with geometric quotients of $X$.

\begin{remark}\label{rem:GITblowup}
In the recursive construction described in Diagram (\ref{eq:blowups}), the $\C^*$-action on the variety $X^\prur_{i,j}$ is of B-type.  

The morphism $b_\prur$ sends the sink $Y^\prur$ of $X^\prur_{i,j}$  onto the sink of $X_{i,j}$, which is, by construction, the geometric quotient $\GX_{i,i+1}$ of $X$; similarly  the contraction $X^\prur_{i,j}\to X_{i+1,j}$ sends $Y^\prur$ onto the geometric quotient $\GX_{i+1,i+2}$. 

Since $X^\prur_{i,j}$ is the smooth blowup of $X_{i,j}$ along the strict transform of $B^+_{i+1}$ into $X_{i,j}$, it follows that $Y^\prur$ can be described as the smooth blowup of $\GX_{i,i+1}$ along the subvariety $V^+(Y_{i+1})$ parametrizing $\C^*$-orbits in $X$ with %sink at $Y_i$ and 
source at $Y_{i+1}$. In a similar way, the source $Y^\prul$ of $X^\prul_{i,j}$ is the smooth blowup of $\GX_{j-1,j}$ along $V^-(Y_{j-1})$.  
\end{remark}

%%%%%%%%%%%%%%%%%%%%%%%%%%%%%%%%%%%%%%%%%%%%%%%%%%%%%%%%%%%%%%%%%%%%%%%%%%%%%%%
%%%%%%%%%%%%%%%%%%%%%%%%%%%%%%%%%%%%%%%%%%%%%%%%%%%%%%%%%%%%%%%%%%%%%%%%%%%%%%%

\section{Chow quotients}\label{sec:Chow}

In this section we consider a nontrivial $\C^*$-action on a polarized pair $(X,L)$ and we  study the Chow quotient of $X$ by $\C^*$, a concept due to Kapranov who, in his seminal work \cite{Kap}, defined it and identified it with a moduli of pointed curves when $X$ is a Grassmannian endowed with an action of a complex torus. Slightly different parameter spaces for $\C^*$-invariant cycles were previously proposed  by Bia{\l}ynicki-Birula and Sommese, \cite{BBSo}, inspired by a construction by Fujiki (cf. \cite{Fj}).

Following Kapranov's definition,  we may find a nonempty open set $V\subset X$ of points $x$ such that the cycles $\overline{\C^*\cdot x}$ belong to a single homology class in $X$. By shrinking $V$, if necessary, the geometric quotient of $V$ by the action of $\C^*$ exists. Moreover, taking closures in $X$ of the orbits of points of $V$, we obtain an algebraic irreducible family of rational curves in $X$, therefore a morphism $\psi:V/\C^*\to \Chow(X)$.

\begin{definition}
With the above notation, the {\em Chow quotient} of $X$ by the action of $\C^*$ is defined as the closure:
$$
\overline{\CX}:=\overline{\psi(V/\C^*)}\subset \Chow(X).
$$
We will denote by $\CX$ the normalization of $\overline{\CX}$, and call it the {\em normalized Chow quotient} of $X$ by the $\C^*$-action. 
The pullback to $\CX$ of the universal family of $\Chow(X)$ will be denoted by $p:\cU\to \CX$, with evaluation morphism $q:\cU\to X$.
\end{definition}

The  $\C^*$-action on $X$ extends naturally to an action on $\Chow(X)$ and the corresponding universal family of cycles. This action restricts to the trivial action on $V/\C^*$, hence to the trivial action on $\CX$. Consequently we have an action on $\cU$, so that the maps $p:\cU\to\CX$, $q:\cU\to X$ are $\C^*$-equivariant, and every cycle parametrized by $\CX$ is $\C^*$-invariant. 

\begin{example}\label{ex:toric2}
In the case of $\C^*$-actions on toric varieties (see Example \ref{ex:toric}), Chow quotients are known to be toric varieties, that have been described combinatorially by Kapranov, Sturmfels and Zelevinsky in \cite{KSZ}. 
\end{example}

Throughout the rest of the paper we will assume that  the  $\C^*$-action is  equalized and that $X$ is smooth. Let us start by observing that in this case every $\C^*$-invariant cycle parametrized by $\CX$ is a chain of rational curves with simple normal crossings
(see \cite[Lemma 3.12]{MMW} for a similar statement for the parameter space of stable maps):

\begin{lemma}\label{lem:reduced}
Let $(X,L)$ be a polarized pair, with $X$ smooth, admitting an equalized $\C^*$-action. 
Let $z\in \CX$ be any point, and let $Z=p^{-1}(z)$ be the corresponding cycle. 
Then $Z$ is reduced, and its dual graph 
 is of type $\DA$. 
\end{lemma}
\begin{proof}
Note that $Z$ is necessarily a connected cycle meeting the sink and the source, of $L$-degree equal to the bandwidth $\delta=a_r$ of the action, by Lemma \ref{lem:AMFM}. If $Z_j$ is an irreducible component of $Z$, denoting by $z_j^-,z_j^+$ its sink and source, the equalization assumption implies that the $L$-degree of the reduced scheme of $Z_j$ is equal to $\mu_L(z_j^+)-\mu_L(z_j^-)$ (again by Lemma \ref{lem:AMFM}), and this number is of type $a_{i}-a_{j}$ for some $i,j\in\{0,\dots,r\}$. We conclude by summing on the components of $Z$.
\end{proof}

\begin{remark}\label{rem:ratcomp}
Note that the proof of Lemma \ref{lem:reduced} also implies that an irreducible component $Z_k$ of $Z=p^{-1}(z)$  is a rational curve of degree $\mu_L(z_k^+)-\mu_L(z_k^-)$, where $z_k^-,z_k^+$ denote the sink and the source of $Z_k$, respectively.  In particular, there are no components of $Z$ contained in a fixed-point component.
\end{remark}

A consequence of the reducedness given by Lemma \ref{lem:reduced} is the following statement, which can be obtained by applying standard arguments on parameter spaces:

\begin{lemma}\label{lem:flat}
Let $(X,L)$ be a polarized pair, with $X$ smooth,  with an equalized $\C^*$-action.   
Then the universal family $p:\cU\to \CX$  is flat.
\end{lemma}

\begin{proof}
Let $Z=p^{-1}(z)$ be any fiber of $p$. By Lemma \ref{lem:reduced}, the components of $Z$ are all reduced, so that $Z$ is normal at the general point of each component and, since $\CX$ is normal, we may apply \cite[Theorem~I.6.6]{kollar} to claim that $p$ is flat at the general point of any component of $Z$. We then conclude by applying \cite[Theorem~I.7.3.1]{kollar}. 
\end{proof}

\begin{remark}\label{rem:Hilb}
As a consequence of the flatness of $\cU\to \CX$, we obtain a morphism $\CX\to \Hilb(X)$ such that $\cU$ is the pullback to $\CX$ of the universal family over $\Hilb(X)$. By shrinking the open subset $V$ introduced above, we can assume that the image of $\CX\to \Hilb(X)$ is the closure of the image of the natural map from $V/\C^*$ to $\Hilb(X)$, whose normalization $\HX$ we call the {\em Hilbert quotient} of the action. On the other hand we have a natural morphism from $\Hilb(X)$ to $\Chow(X)$ (cf. \cite[Theorem I.6.3]{kollar}), so it turns out that $\CX$ coincides with the Hilbert quotient of $X$. The idea of considering a flat family of $\C^*$-invariant cycles, whose general element is the closure of a general orbit --in the complex geometric setting-- goes back to Fujiki, \cite{Fj} (see also \cite[Theorem~0.1.2]{BBSo}).
\end{remark}

\begin{remark}\label{rem:chblow}
Given an equalized $\C^*$-action on $X$, every cycle parametrized by $\CX$ has a unique lift-up to the blowup $X^\flat$  of $X$ along the sink and the source. In particular  $\CX \simeq \CX^\flat$.
\end{remark}

\begin{proposition}\label{prop:comp}
If the action is of B-type and faithful, then  $\CX$ is the normalization of the component of the Chow/Hilbert scheme containing the class of the closure of a general orbit of the action.
\end{proposition}
\begin{proof}
The faithfulness implies that the action is equalized at the sink and the source (see \cite[Remark 3.2]{WORS1}).
If $C\iso\mathbb{P}^1$ is the closure of a general orbit, then we can use \cite[Lemma 2.16]{WORS1} to compute the splitting type of the normal bundle of $C$ in $X$. Since the weights of the $\C^*$-action on $\cN_{C/X}$ are $0^{(\dim X-1)}$ both at the sink and at the source of $C$,
it follows that $\cN_{C/X}\iso\cO_{\PP^1}^{\oplus n-1}$ and therefore the component of $\Hilb(X)$ containing the point $[C]$ is smooth at $[C]$ and of dimension $n-1 = \dim \CX$.
\end{proof}

%%%%%%%%%%%%%%%%%%%%%%%%%%%%%%%%%%%%%%%%%%%%%%%%%%%%%%%%%%%
\subsection{The universal bundle of a $\C^*$-action}\label{ssec:bundle}

Let $(X,L)$ be a polarized pair, with $X$ smooth, supporting an equalized $\C^*$-action.
We will now construct a vector bundle on $\CX$, whose fibers represent, roughly speaking, the linear spans  of the cycles parametrized by $\CX$ in a certain embedding of $X$. This will allow us to identify $\CX$ with a  subvariety of a Grassmannian. 

\begin{lemma}\label{lem:univbdl}
Let $(X,L)$ be a polarized pair, with $X$ smooth, supporting an equalized $\C^*$-action.  
Then $p_*(q^*(L))$ is locally free of rank $\delta+1$.
\end{lemma}

\begin{proof}
We use the equalization of the action to ensure, by Lemma \ref{lem:flat}, that $p$ is flat. In particular, $\chi(C,L)$ is the same for every cycle $C$  of the family parametrized by $\CX$, considered as a scheme with its reduced structure. We will show that $\HH^1(C,L)=0$ for every $C$, so that $\dim\HH^0(C,L)$ will be equal to $\delta+1$, and we may then apply \cite[Corollary~III.12.9]{Ha} to obtain the claimed result. 

In order to do so, since the dual graph of $C$ is of type $\DA$ by Lemma \ref{lem:reduced}, we may  number the irreducible components  of a given element $C$ of $\CX$ as $C_1,\dots,C_k$  so that $C_i$ intersects (transversally) $C_{i+1}$ in a point $p_i$, $i=1,\dots,k-1$, and $\{\mu_L(p_i)\}$ is strictly increasing. Note that we have an exact sequence:
\begin{equation}
0\to \HH^0(C,L)\lra \bigoplus_{i=1}^{k}\HH^0(C_i,L_{|C_i})\lra \bigoplus_{i=1}^{k-1}\C_{p_i}\lra \HH^1(C,L)\to 0.
\label{eq:irredcomp}
\end{equation}
Then, since $L$ is ample, one may easily check that the restriction map
$$
\bigoplus_{i=1}^{k}\HH^0(C_i,L_{|C_i})\lra \bigoplus_{i=1}^{k-1}\C_{p_i}
$$
is surjective, hence  $\HH^1(C,L)$ is equal to zero. This completes the proof.
\end{proof}

\begin{definition}
The vector bundle $$\cE:=p_*(q^*(L))$$ will be called the {\em universal bundle} associated to the $\C^*$-action on $(X,L)$.
\end{definition}

\begin{lemma}
In the setting of Lemma \ref{lem:univbdl}, we have a $\C^*$-invariant decomposition $$\cE=\bigoplus_{u=0}^\delta \cL_u,$$ where $\cL_u$ is a line bundle on which $\CC^*$ acts with weight $u$.
\end{lemma}

\begin{proof}
The action of $\C^*$ on $X$ extends naturally to a linearization on the vector bundle $\cE$, that we denote by $\mu$. On the fiber $C$ over a point $c\in \CX$ this is the induced action of $\C^*$ on the $(\delta+1)$-dimensional vector space $\HH^0(C,L_{|C})$. For $c\in\CX$ general,  by Lemma \ref{lem:AMFM}, $\HH^0(C,L_{|C})$ is $\C^*$-equivariantly isomorphic to $\HH^0(\P^1,\cO_{\P^1}(\delta))$ equipped with the natural action of $\C^*$, which splits as a direct sum of $1$-dimensional eigenspaces, associated to the weights $0,1,\dots,\delta$. 

For every $u\in\{0,1,\dots,\delta\}$ we consider the set $$\cL_u:=\bigcap_{t\in \C^*}\ker\big(\mu(t,\cdot)-u\id:\cE\to \cE\big),$$ 
so that over every point $c$ we have an inclusion $\bigoplus_{u}\cL_{u,c}\subset \cE_c$. If $c$ is general we have seen that the above inclusion is an equality and $\dim\cL_{u,c}=1$ for every $u$. Since by semicontinuity the dimension of $\cL_{u,c}$ at the general point is the minimal one, it follows that $\dim\cL_{u,c}$ is constant in $c$, that is, $\cL_u$ is a vector subbundle of $\cE$ of rank one, and we have a decomposition $\cE=\bigoplus_u\cL_u$. 
\end{proof}

\begin{definition}\label{def:noH1}
We will say that $(X,L)$ is {\em Chow-projectively normal} with respect to the  $\C^*$-action ({\em CPN}, for short) if $L$ is very ample, and
$$
\HH^1(X,L\otimes I_C)=0,\,\,\mbox{for every cycle $C$ of the family parametrized by $\CX$}.
$$
\end{definition}

Note that this condition implies that the restriction of global sections $\HH^0(X,L)\to\HH^0(C,L_{|C})$ is surjective for every $C$, so that the evaluation map 
$$\HH^0(X,L)\otimes \cO_{\CX}=\HH^0(\CX,\cE)\otimes \cO_{\CX}\lra\cE$$
is surjective. Furthermore, since this map is $\C^*$-equivariant, it provides surjections:
\begin{equation}\label{eq:surjec}
\HH^0(X,L)_u\otimes \cO_{\CX}\lra\cL_u,
\end{equation}
for every weight $u\in\{0,1,\dots,\delta\}$. 

\begin{remark}\label{rem:ratnormal0}
Note that, in the setting of Lemma \ref{lem:univbdl}, since the universal family parametrized by $\CX$ is flat 
(cf.  Lemma \ref{lem:flat}), by semicontinuity 
there exists an integer $m_0$ such that for $m\geq m_0$ we have $\HH^1(X,L^{\otimes m}\otimes I_C)=0$ for every cycle $C$ of the family parametrized by $\CX$.  
In other words if $L$ is ample, $X$ is smooth and the action of $\C^*$ on $(X,L)$ is equalized, then $(X,mL)$ is CPN for $m\gg 0$. 
\end{remark}

\begin{remark}\label{rem:ratnormal}
The CPN hypothesis can be read geometrically as follows. Via the embedding $X\hookrightarrow \P(\HH^0(X,L))$, every $1$-cycle $C$ parametrized by $\CX$ is mapped to a reduced $\C^*$-invariant rational cycle of $L$-degree $\delta$ in $\P(\HH^0(C,L_{|C}))$, on which the $\C^*$-action acts with weights $0,1,\dots,\delta$. In particular, if $C_1,\dots, C_k$ are the irreducible components of $C$, with the notation of the proof of Lemma \ref{lem:univbdl}, from the exact sequence (\ref{eq:irredcomp}) it follows that the restriction map $\HH^0(X,L)\to \HH^0(C_i,L_{|C_i})$ is surjective, and so every component $C_i$ embeds into $\P(\HH^0(C_i,L_{|C_i}))\subset \P(\HH^0(X,L))$ as a rational normal curve. See Figure \ref{fig:CPN}.
\end{remark}

\begin{figure}[h!]
\begin{tikzpicture}
  \node at (0,0) {\includegraphics[width=8cm]{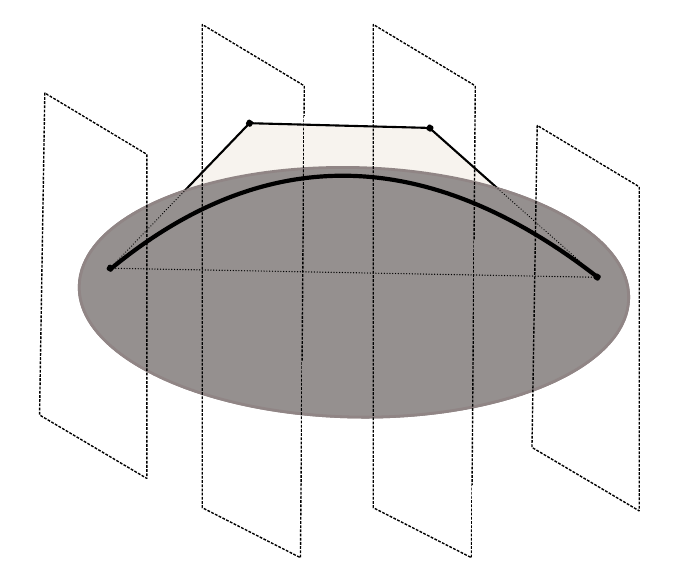}};
  \node at (-0.1,-1.8) {$X$};
  \node at (1.9,2.9) {$\P(\HH^0(X,L)_u)$};
  \node at (-1.3,0.85) {$C$};
  \node at (-0.1,1.62) {$\P(\HH^0(C,L_{|C}))$};
\end{tikzpicture}
\caption{A $\C^*$-invariant curve $C$ is a rational normal curve in its linear span $\P(\HH^0(C,L_{|C}))$, and this space meets each fixed point subspace $\P(\HH^0(X,L)_u)$ at a point.\label{fig:CPN}}
\end{figure}
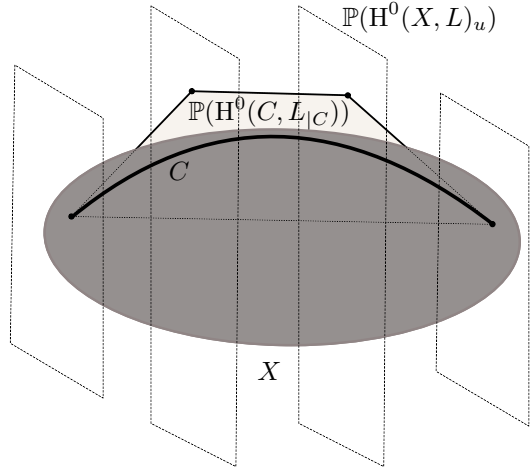

\begin{lemma}\label{lem:CPNprop}
Let $(X,L)$ be a polarized pair, with $X$ smooth, supporting an equalized $\C^*$-action. If $(X,L)$ is CPN with respect to the $\C^*$-action, then we have an embedding $\imath:\cU\hookrightarrow\P(\cE)$, an isomorphism $\HH^0(\CX,\cE)\simeq \HH^0(X,L)$, and the morphism $q:\cU\to X$  composed with the embedding $X\hookrightarrow \P(\HH^0(X,L))$ factors via the morphism $\P(\cE)\to \P(\HH^0(X,L))$ induced by the evaluation of global sections. In other words, we have a commutative diagram:

$$
\xymatrix{\CX&\cU\ar[d]^q\ar[l]^p\ar@{^(->}+<8pt,0pt>;[r]_{\imath}&\P(\cE)\ar@/_12pt/[ll]\ar[d]\\
&X\ar@{^(->}+<8pt,0pt>;[r]&\P(\HH^0(X,L))}
$$

\end{lemma}

\begin{proof}
The natural surjection:
$$
p^*\cE=p^*p_*q^*L\to q^*L
$$
provides a morphism $\imath:\cU\to\P(p^*\cE)\to \P(\cE)$, that, composed with the projection $\P(\cE)\to \CX$ is equal to $p$. Since over every $c\in\CX$ the map $\imath$ coincides with the embedding of the corresponding cycle $C$ into $\P(\HH^0(C,L_{|C}))$, we may conclude that $\imath$ is an embedding. 

Moreover $\HH^0(\CX,\cE)\simeq\HH^0(\CX,p_*q^*L)\simeq\HH^0(\cU,q^*L)$; since $q\colon \cU\to X$ is birational and $X$ is normal, by Zariski's Main Theorem $q$ has connected fibers and so $\HH^0(\cU,q^*L)\simeq \HH^0(X,L)$ and the evaluation of global sections induces $\P(\cE)\to \P(\HH^0(X,L))$, and this map obviously commutes with the arrows in the diagram of the statement.
\end{proof}

\begin{corollary}\label{cor:chowgrass}
Let $(X,L)$ be a polarized pair, with $X$ smooth, supporting an equalized $\C^*$-action. Assume that $(X,L)$ is CPN with respect to the $\C^*$-action. Let $\delta$ be the bandwidth of the induced action on $\P(\HH^0(X,L))$. Then we have a  
morphism $$\gamma:\CX\lra  
\G(\delta,\P(\HH^0(X,L))),$$
such that the pullback via $\gamma$ of the Pl\"ucker line bundle is $\det(\cE)=\bigotimes_{u=0}^\delta \cL_u$.
\end{corollary}

\begin{proof}
This follows from the fact that, by Lemma \ref{lem:univbdl}, the vector bundle $\cE$ has rank $\delta+1$ and it is globally generated by $\HH^0(X,L)$. In particular, the composition of  $\gamma$ with the Pl\"ucker embedding of $\G(\delta,\P(\HH^0(X,L))$ into $\P\big(\bigwedge^{\delta+1}\HH^0(X,L)\big)$ is the morphism associated with the line bundle $\det (\cE)$.
\end{proof}

\begin{remark}\label{rem:Sigma}
Note that the image of $\gamma$ is contained in the Segre variety $\Sigma(\HH^0(X,L))$. In fact, under our assumptions, for every $u\in\{0,1,\dots,\delta\}$ we have surjective maps $\HH^0(X,L)_u\otimes \cO_{\CX}\to \cL_u$. Geometrically this means that for every point $c\in\CX$, the projective space $\P(\cE_c)$ is the linear span of the points $\P(\cL_{u,c})\in\P(\HH^0(X,L)_u)$. Then the claim follows by the definition of $\Sigma(\HH^0(X,L))$ -- see Equation (\ref{eq:Sigma}). 
\end{remark}

\begin{proposition}\label{prop:normal}
In the situation of Corollary \ref{cor:chowgrass}, the map $\gamma$ is the normalization of its image.
\end{proposition}

\begin{proof}
Since $\CX$ is the normalization of the Chow quotient   $\overline{\CX}\subset\Chow(X)$, it is enough to show that the morphism $\gamma$ factors via an injective map from $\overline{\CX}$ to $\Sigma(\HH^0(X,L))$. In other words, we need to prove that different $\C^*$-invariant cycles (corresponding to different points in $\overline{\CX}$) have different linear spans.

Let $C,C'\in\overline{\CX}$ be two $\C^*$-invariant cycles, and denote by $C_1,\dots ,C_t
$, $C'_1,\dots ,C'_{t'}$ their irreducible components, ordered by the weights of their extremal fixed points as in the proof of Lemma \ref{lem:univbdl}. Let $P_u, P'_u\in\P(\HH^0(X,L)_u)$ be the intersections of the linear spans of $C,C'$ with $\P(\HH^0(X,L)_u)$, for every weight $u=0,1,\dots,\delta$.

Assume now that the linear spans of $C$, $C'$ coincide; equivalently,  $P_u=P'_u$ for every $u$. Note that the first irreducible component $C_1$ of $C$ can be parametrized as $\sum_{u=0}^{k}t^uP_u$, $t\in\C$, in a neighborhood of $P_0$; in particular, its tangent line at $P_0$ is $\langle P_0,P_1\rangle$. Since the same happens for $C'_1$, and every $1$-dimensional orbit is completely determined by the tangent direction  at the sink, by the Bia{\l}ynicki-Birula theorem, the fact that $\langle P_0,P_1\rangle=\langle P'_0,P'_1\rangle$ tells us that $C_1=C'_1$. Recursively, we prove that $C_i=C'_i$ for every $i$, that is $C=C'$.
\end{proof}

%%%%%%%%%%%%%%%%%%%%%%%%%%%%%%%%%%%%%%%%%%%%%%%%%%%%%%%%%%%%%%%%%%%%%%%%%%%%%%%%%%%%%%%%%%%%%%%%
%%%%%%%%%%%%%%%%%%%%%%%%%%%%%%%%%%%%%%%%%%%%%%%%%%%%%%%%%%%%%%%%%%%%%%%%%%%%%%%%%%%%%%%%%%%%%%%%

\section{Proof of the main statement}\label{sec:main}

This section is devoted to the proof of Theorem \ref{thm:main}. First of all we will show the existence of morphisms among the normalized Chow quotients $\CX_{i,j}$ fitting in the commutative diagram (\ref{fig:Chow0}), then we will prove (Proposition \ref{prop:blowup}) that the morphisms in the diagram are blowups, whose centers are described in Remark \ref{rem:center}. We finish the section discussing the smoothness of $\CX$. 

Throughout the section $(X,L)$ will  be a polarized pair, with $X$ smooth, endowed with an equalized $\C^*$-action of criticality $r$ such that  $\codim(B^\pm(Y),X)>1$ for every inner fixed point component $Y\in\cY^\circ$. 
For every small $\QQ$-factorial modification $X_{i,j}$ as in Proposition \ref{prop:WORS3} we choose a $\Q$-divisor $L_{i,j}=L(\tau_{i,j}^-,\tau_{i,j}^+) \in\Int(N_{i,j})$, which is  ample on $X_{i,j}$.

\begin{lemma}\label{lem:manyCPNs}
There exists an integer $m$ such that $(X,mL)$ and every $(X_{i,j},mL_{i,j})$ are CPN  with respect to the $\C^*$-action. Moreover, for every $0 \le i \le j\le r$ we have $\C^*$-equivariant isomorphisms:
$$
\HH^0(X_{i,j},mL_{i,j})\simeq \HH^0(X,mL)_{m\tau_{i,j}^-,m\tau_{i,j}^+}.
$$  
\end{lemma}

\begin{proof}
The first part of the statement follows from Remark \ref{rem:ratnormal0}. Moreover, since $X_{i,j}$ and $X^\flat$ are isomorphic in codimension one, for $m$ large enough and divisible, we have $\C^*$-equivariant isomorphisms $\HH^0(X_{i,j},mL_{i,j})\to \HH^0(X^\flat,mL_{i,j})$. Then we conclude the proof observing that,  by \cite[Lemma~2.5]{WORS3}, we have $\C^*$-equivariant isomorphisms $\HH^0(X^\flat,mL_{i,j})\simeq $ $\HH^0(X,mL)_{m\tau_{i,j}^-,m\tau_{i,j}^+}$.
\end{proof}

As a consequence of  Lemma \ref{lem:manyCPNs}, replacing $L$ with a sufficiently large and divisible multiple, we may work, without loss of generality, in the following: 

\begin{setup}\label{set:CPNs}
 $(X,L)$ is a polarized pair, with $X$ smooth, endowed with an equalized $\C^*$-action of criticality $r$ such that  $\codim(B^\pm(Y),X)>1$ for every inner fixed point component $Y\in\cY^\circ$, $L_{i,j} \in \Pic(X_{i,j})$ ample, for $0 \le i \le j \le r$ and
the polarized pairs  $(X,L)$ and $(X_{i,j},L_{i,j})$ are CPN with respect to the $\C^*$-action.
\end{setup}
In particular, for $0 \le i \le j \le r$, we have embeddings
$X_{i,j}\hooklongrightarrow \P(\HH^0(X_{i,j},L_{i,j}))$ 
which fit in the commutative diagram:
$$
\xymatrix@C=15mm@R=8mm{X\ar@{^{(}->}+<10pt,0pt>;[r]\ar@{-->}[d]&\P(\HH^0(X,L))\ar@{-->}[d]^{\phi_{i,j}}\\
X_{i,j}\ar@{^{(}->}+<12pt,0pt>;[r]&\P(\HH^0(X_{i,j},L_{i,j}))}
$$
In this diagram $\phi_{i,j}$ denotes the linear projection induced by the monomorphism:
$$
\HH^0(X_{i,j},L_{i,j})\simeq \HH^0(X,L)_{\tau_{i,j}^-,\tau_{i,j}^+}\hooklongrightarrow \HH^0(X,L),
$$
and the left-hand-side vertical map denotes the composition of the blowup map and the pruning map (see Section \ref{sec:prun}):
$$\xymatrix@C=15mm{X\ar@{-->}[r]^{\beta^{-1}}&X^\flat\ar@{-->}[r]^{\varphi_{i,j}}&X_{i,j}}$$
Abusing notation, we will denote this composition by $\phi_{i,j}$, as well.
Let us set:
$$
V:=\HH^0(X,L),\qquad V_{i,j}:=\HH^0(X_{i,j},L_{i,j});
$$
we will use the notation introduced in Section \ref{ssec:grass}, regarding Grassmannians of linear spaces in $\P(V)$ and $\P(V_{i,j})$, and we will consider the normalized Chow quotients $\CX_{i,j}$ of the varieties $X_{i,j}$. Note first that:
\[\CX_{i,i+1}\simeq \GX_{i,i+1},\quad \mbox{for every }i=0,\dots,r-1.\]  
Moreover, by Corollary \ref{cor:chowgrass}, for every $i<j$ we have a morphism 
$$
\gamma_{i,j}:\CX_{i,j}\to \G(\delta_{i,j},\P(V_{i,j})),
$$
where $\delta_{i,j}$ denotes the bandwidth of the induced $\C^*$-action on $(X_{i,j}, L_{i,j})$. The image of the above map is contained in the Segre variety $\Sigma(V_{i,j})$ by Remark \ref{rem:Sigma}. The next Lemma shows that the maps $\gamma$, $\gamma_{i,j}$ are compatible with the projection $\pi_{i,j}\colon \Sigma(V)\to \Sigma(V_{i,j})$ induced by the projection $\phi_{i,j}$ (see Section \ref{ssec:grass}) and the natural map sending invariant cycles in $X$ to invariant cycles in $\CX_{i,j}$; moreover, we show that this map is a morphism.

\begin{lemma}\label{lem:projC}
Let $(X,L)$ be as in Setup \ref{set:CPNs}. For  $0 \le i \le j \le r$, the birational map $\phi_{i,j}:X\dashrightarrow X_{i,j}$ induces a morphism $\rho_{i,j}:\CX\to \CX_{i,j}$ fitting in the commutative diagram:
$$
\xymatrix{\CX\ar[rr]\ar[d]_{\rho_{i,j}}&&\Sigma(V)\ar[d]^{\pi_{i,j}}\\
\CX_{i,j}\ar[rr]&&\Sigma(V_{i,j})
}
$$
where $\pi_{i,j}$ denotes the projection between Segre varieties induced by $\varphi_{i,j}$.
\end{lemma}

\begin{proof}
Let us consider the composition $\pi_{i,j}\circ\gamma:\CX\to \Sigma(V_{i,j})$. Since the projection via $\phi_{i,j}$ of a general --irreducible-- $\C^*$-invariant $1$-cycle in $X$ is an irreducible $\C^*$-invariant $1$-cycle in $X_{i,j}$, it follows that $\pi_{i,j}(\gamma(\CX))$ is contained in $\gamma_{i,j}(\CX_{i,j})$. By Proposition \ref{prop:normal} $\CX_{i,j}$ is the normalization of its image via $\gamma_{i,j}$, thus the map $\pi_{i,j}\circ\gamma$ factors via $\CX_{i,j}$, as stated.
\end{proof}

\begin{remark}\label{rem:projC1}
Denoting by $\delta_{i,j}$ the bandwidth of the induced $\C^*$-action on $(X_{i,j}, L_{i,j})$, the restriction to $\Sigma(V_{i,j})$ of the Pl\"ucker embedding of the Grassmannian $\G(\delta_{i,j},\HH^0(X_{i,j},L_{i,j}))$  is the Segre embedding of that variety, identified with $\prod_u \P(\HH^0(X,L)_{u}) $ where $u=\tau_{i,j}^-, \dots, \tau_{i,j}^+$. Moreover,  
$\pi_{i,j}:\Sigma(V)\to \Sigma(V_{i,j})$ can be identified with the canonical projection 
\[\textstyle \prod_{u=0}^{\delta}\P(\HH^0(X,L)_{u}) \to \prod_{u=\tau_{i,j}^-}^{\tau_{i,j}^+}\P(\HH^0(X,L)_{u}).\] Together with Lemma \ref{lem:projC} and Remark \ref{rem:Sigma}, this implies that the morphism $\pi_{i,j}\circ\gamma:\CX\to \Sigma(V_{i,j})$, and subsequently its normalization $\rho_{i,j}:\CX\to \CX_{i,j}$,  is given by the line bundle $\bigotimes_u 
{\cL_u}$, where $u=\tau_{i,j}^-, \dots, \tau_{i,j}^+$. 
\end{remark}

\begin{remark}\label{rem:projC2} Note that, by Remark \ref{rem:chblow}, $\CX=\CX^\flat$, and that $X_{0,r}=X^\flat$ by definition (see proof of Proposition \ref{prop:WORS3}), so $\CX=\CX_{0,r}$. Moreover  the above argument can be applied to any modification $X_{i,j}$, obtaining morphisms 
$$\rho^{i,j}_{k,\ell}:\CX_{i,j}\lra\CX_{k,\ell},$$ 
whenever $i\leq k\le \ell\leq j$. Abusing notation, we will call these morphisms {\em pruning maps}, and denote 
$$
\prul:=\rho^{i,j}_{i,j-1},\qquad \prur:=\rho^{i,j}_{i+1,j}
$$
(called {\em elementary pruning maps}). As in Remark \ref{rem:GITblowup}, $s$ and $d$ commute, so every pruning map $\rho^{i,j}_{k,\ell}$ can be written as $\rho^{i,j}_{k,\ell}=\prul^{i-k}\circ \prur^{j-\ell}$.

\begin{equation}\label{fig:Chow}
\begin{tikzcd}[
  column sep={2.9em,between origins},
  row sep={3.2em,between origins}]
&&&&&\CX=\CX_{0,r}\arrow[rd,"\prur"] \arrow[dl,"\prul",labels=above left] &&&&&\\
&&&&\CX_{0,r-1}\arrow[rd,"\prur"] \arrow[dl,"\prul",labels=above left]&&\CX_{1,r}\arrow[rd,"\prur"] \arrow[dl,"\prul",labels=above left]&&&&\\
&&&\CX_{0,r-2} &&\CX_{1,r-1}&&\CX_{2,r}&&&\\
&& \iddots&\ddots& \iddots&\dots&\ddots&\iddots&\ddots&&\\
&\GX_{0,1}\arrow[rd] \arrow[dl]&&\GX_{1,2} \arrow[dl]&&\dots&&\GX_{r-2,r-1}\arrow[rd] &&\GX_{r-1,r}\arrow[rd] \arrow[dl]&\\
\GX_{0,0}&&\GX_{1,1}&&\dots&&\dots&&\GX_{r-1,r-1}&&\GX_{r,r}
\end{tikzcd}
\end{equation}
Then these morphisms fit into the commutative diagram (\ref{fig:Chow}), which extends Diagram (\ref{eq:GITquot}).
\end{remark}

\begin{proposition}\label{prop:blowup}
Let $(X,L)$ be as in Setup \ref{set:CPNs}. Then for every $i<j$ the maps:
$$\xymatrix@C=5mm@R=7mm{&\CX_{i,j}\ar[ld]_{\prul}\ar[rd]^{\prur}&\\\CX_{i,j-1}&&\CX_{i+1,j}}$$ 
are normalizations of blowups.
\end{proposition}
The proof will be done in three steps. First of all, we will consider the case of an action of criticality $r=2$:

\begin{lemma}\label{lem:blowup1}
In the situation of Proposition \ref{prop:blowup}, assume furthermore that $r=2$. Then $\prul:\CX_{0,2}
\to\CX_{0,1}=\GX_{0,1}$,  $\prur:\CX_{0,2}
\to\CX_{1,2}=\GX_{1,2}$ are blowups along smooth centers which are the exceptional loci of the birational maps $\GX_{0,1} \to \GX_{1,1}$, $\GX_{1,2} \to \GX_{1,1}$.
\end{lemma}

\begin{proof} 
We will prove the statement for $\prul$; the case of $\prur$ is analogous. We consider the blowup $X^\prur_{0,2}$ of $X_{0,2}$ along (the strict transform of) $B^+_1$, and its sink $Y^\prur$. By Remark \ref{rem:GITblowup} the variety  $Y^\prur$ is the smooth blowup of  $\GX_{0,1}=\CX_{0,1}$ along the subset parametrizing $\C^*$-orbits in $X_{0,2}$ with %sink at $Y_0$ and 
source at an inner fixed point component of weight $a_1$. It is then enough to show that $Y^\prur\simeq \CX_{0,2}$. 

We note first that, if $Y^\prur_1 \subset X^\prur_{0,2}$ is %the inverse image of 
an inner fixed point component of $X^\prur_{0,2}$, mapping via $b_\prur$ onto a fixed point component $Y$ of weight $a_1$, then $\nu^-(Y^\prur_1)=1$. This implies that, given the closure of any orbit with source $P$ at $Y^\prur_1$, there exists a unique $\C^*$-invariant curve with sink at $P$ (Cf. \cite[Lemma 2.2]{WORS5}). In particular the morphism $\CX^\prur_{0,2}\to Y^\prur=\GX^\prur_{0,1}$ is bijective. By the normality of $\CX^\prur_{0,2}$ and $Y^\prur$, it follows that this map is an isomorphism.

On the other hand, the blowup morphism $b_\prur:X^\prur_{0,2} \to X_{0,2}$ is $\C^*$-equivariant, hence it induces a surjective morphism $\CX^\prur_{0,2}\to \CX_{0,2}$. We claim that this map is also injective and, consequently, an isomorphism. The morphism $b_\prur$ is an isomorphism outside of the union of $\C^*$-invariant curves linking $Y^\prur$ to inner fixed points component of weight $a_1$. It is  then enough to show that for every connected $\C^*$-invariant $1$-cycle $C=C_1+ C_2  \subset X_{0,2}$, with $C_1$ linking the sink to a fixed point component of weight $a_1$, and $C_2$ linking a fixed point component of weight $a_1$ to the source, there exists a unique $\C^*$-invariant connected $1$-cycle $C'=C'_1+ C'_2$ such that  $b_\prur(C'_i)=C_i$. Note that $C'_2$ is necessarily the strict transform of $C_2$ into $X^\prur_{0,2}$, which is unique. On the other hand $C'_1$ and $C_1$ are uniquely determined by their tangent directions at their sources $P'\in X^\prur_{0,2}$, $P=b_\prur(P')$. Since the differential of $b_\prur$ maps $\cN^+(Y'_1)_{P'}$ isomorphically to $\cN^+(Y_1)_P$, the uniqueness of $C'_1$ follows.
\end{proof}

As a consequence, we get that, in the Diagram (\ref{fig:Chow}), all the pruning maps:
\begin{equation}\label{eq:firstblowups}
\begin{tikzcd}[
  column sep={2.9em,between origins},
  row sep={3.2em,between origins},
]
&\CX_{0,2}\arrow[rd,"\prur"]\arrow[dl,"\prul",labels=above left]&&\CX_{1,3} \arrow[dl,"\prul",labels=above left]&&\dots&&\CX_{r-3,r-1}\arrow[rd,"\prur"] &&\CX_{r-2,r}\arrow[rd,"\prur"] \arrow[dl,"\prul",labels=above left]&\\
\GX_{0,1}&&\GX_{1,2}&&\dots&&\dots&&\GX_{r-2,r-1}&&\GX_{r-1,r}
\end{tikzcd}
\end{equation}
are smooth blowups. The next statement tells us that the rest of the maps in Diagram (\ref{fig:Chow}) are, up to normalization, pullbacks of these blowups. 

\begin{lemma}\label{lem:blowup2}
In the situation of Proposition \ref{prop:blowup}, assume  that $r>2$ and consider the following   diagram:
$$
\xymatrix@=5mm{&\CX_{0,r}\ar[ld]_{\prul^{r-2}}\ar[rd]^{\prur}&\\
\CX_{0,2}\ar[rd]_{\prur}&&\CX_{1,r}\ar[ld]^{\prul^{r-2}}\\&\GX_{1,2}&}
$$
Then $\CX_{0,r}$ is the normalization of the fiber product of $\CX_{0,2}$ and $\CX_{1,r}$ over $\GX_{1,2}$.
\end{lemma}

\begin{proof}
For simplicity, we will set $\zeta:=s^{r-2}$.  We will show that the natural map from $\CX_{0,r}$ to the fiber product of $\CX_{0,2}$ and $\CX_{1,r}$ over $\GX_{1,2}$ is a bijection by showing that, for every $c\in \CX_{1,r}$, $\zeta$ induces a bijection between $\prur^{-1}(c)$ and $\prur^{-1}(\zeta(c))$. 

 We will describe the fiber $\prur^{-1}(c)$, for an element $c\in \CX_{1,r}$. 
We choose a preimage $\gamma\in \prur^{-1}(c) \subset \CX_{0,r}$ and we denote the irreducible components of the corresponding $1$-cycles $C$ and $\Gamma$ by $C_1,\dots, C_k\subset X_{1,r}$ and
 $\Gamma_1,\dots, \Gamma_{l}\subset X_{0,r}$,  respectively, ordered by the weights of their extremal fixed points. 
 In particular, $\Gamma_1$ meets the sink of $X_{0,r}$, and we denote by $P$ its source. 

 Let us consider another element $\delta \in \prur^{-1}(c)$, which corresponds to a $1$-cycle $\Delta=\Delta_1+\dots+\Delta_{m}$, whose components are numbered as above. We now distinguish two cases, according to the weight $\mu_L(P)$.  

If $\mu_L(P)>a_1$,  then, by Lemma \ref{lem:projC}, we have $m=l=k$ and $\varphi_{1,r}(\Gamma_i)=\varphi_{1,r}(\Delta_i)=C_i$ for every $i=1, \dots, k$. Since $\mu_L(P)>a_1$, then  $\Gamma_2,\dots,\Gamma_{k}$ and $\Delta_2,\dots,\Delta_{k}$ are contained in the open set where $\varphi_{1,r}$ is an isomorphism, and so we conclude that $\Gamma_j=\Delta_j$ for every $j\geq 2$. In particular the source of $\Gamma_1$ is $P$ and, since $\varphi_{0,r}$ is a $\C^*$-equivariant isomorphism in a neighborhood of this point, it follows that $\Gamma_1=\Delta_1$. Therefore $\Gamma=D'$ and $\gamma=\delta$. In particular $\prur^{-1}(\zeta(c))$ is a point.

If $\mu_L(P)=a_1$, a similar argument shows that $l=m=k+1$, and that $\Gamma_j=\Delta_j$ for $j\geq 2$. On the other hand, in this case $\Delta_1$ can be any irreducible $\C^*$-invariant (not fixed) curve with source at $P$. Since the action is equalized, it follows that $\prur^{-1}(c)$ is bijective to  the projective space $\P(\cN^{+}(Y_1)^{\vee}_{P})$ (see Remark \ref{rem:normalquot}). Since $\varphi_{0,2}$ is an equivariant isomorphism in a neighborhood of $P$, it follows that $\zeta$ maps $\P(\cN^{+}(Y_1)^{\vee}_{P})$ isomorphically onto $\prur^{-1}(\zeta(c))$. This completes the proof.
\end{proof}

\begin{proof}[Proof of Proposition \ref{prop:blowup}] 
The blowup of a variety along a closed subset is uniquely determined by its universal property, namely that the inverse image of the closed subset is a Cartier divisor, and that it is a final object with this property (\cite[II, Proposition 7.14]{Ha}). Moreover, it is known (see \cite[Proposition IV-21, Lemma IV-41]{EH00}), that the fiber product of two blowup maps $b_1:M_1\to M$, $b_2:M_2\to M$ of a variety $M$ along two subvarieties $Z_1,Z_2$ equals the blowup of $M_1$ along $b_1^{-1}(Z_2)$ and the blowup of $M_2$ along $b_2^{-1}(Z_1)$. 

In our case, the exceptional locus of $\prur:\CX_{0,r}\to \CX_{1,r}$ is the inverse image of the exceptional locus of $\prur:\CX_{0,2}\to \CX_{1,2}$, which is a Cartier divisor by Lemma \ref{lem:blowup1}. The universality then follows easily from Lemma \ref{lem:blowup2}, by the universal property of the Cartesian square and the normalization.

A similar argument tells us that also the map $\prul:\CX_{0,r}\to \CX_{0,r-1}$ is the normalization of a blowup of $\CX_{0,r-1}$, and that this map is the pullback to $\CX_{0,r-1}$ of $\prul:\CX_{r-2,r}\to \CX_{r-2,r-1}$. Applying this to the Chow quotients of $X_{i,j}$ we obtain that also the morphisms $\prul:\CX_{i,j}\to \CX_{i,j-1}$, $\prur:\CX_{i,j}\to \CX_{i+1,j}$ are normalizations of blowups, and that the two parallelograms in the diagram:
$$
\xymatrix@C=1mm@R=4mm{&&\CX_{i,j}\ar[rd]\ar[ld]&&\\
&\CX_{i,j-1}\ar[rrdd]\ar[ld]&&\CX_{i+1,j}\ar[rd]\ar[lldd]& \\
\CX_{i,i+2}\ar[rd]&&&&\CX_{j-2,j}\ar[ld]\\
&\CX_{i+1,i+2}&&\CX_{j-2,j-1}&} 
$$
are normalizations of Cartesian squares. It easily follows that all the parallelograms (in particular all the rhombuses) in Diagram (\ref{fig:Chow}) are normalizations of Cartesian squares, as well. Now, since all the maps in Diagram (\ref{fig:Chow}) are constructed by successive normalizations of fiber products upon Diagram (\ref{eq:firstblowups}), and all the maps in Diagram (\ref{eq:firstblowups}) are normalizations of  blowups, then  all the maps in Diagram (\ref{fig:Chow}) are normalizations of  blowups. 
\end{proof}

A consequence of Theorem \ref{thm:main} is that $\CX$ parametrizes all the $\C^*$-invariant connected cycles with dual diagram of type $\DA$ with extremal fixed point components in the sink and the source of $X$: 

\begin{corollary}\label{cor:CXcontainsall}
Let $(X,L)$ be a smooth polarized pair endowed with an equalized $\C^*$-action of criticality $r$, such that $\nu^\pm(Y) \ge 2$ for every fixed point component of weight $a_i$ with $2 \le i \le r-2$.  Let $C$ be a $\C^*$-invariant $1$-dimensional reduced subscheme with irreducible components $C_1,\dots,C_k$ satisfying that 
\begin{itemize}
\item[(C1)] every $C_j$ is the (smooth) closure of a $1$-dimensional orbit; 
\item[(C2)] the singular points of $C$ are the (transversal) intersections $P_j:=C_j\cap C_{j+1}$, are fixed points of weights $a_{i_j}$, and  $0<i_1<\dots<i_{k-1}<r$; % (where $r$, as usual, denotes the criticality of the action); 
\item[(C3)] every $P_j$ is the source of $C_j$ and the sink of $C_{j+1}$, for $j=1,\dots,k-1$;  
\item[(C4)] the sink of $C_1$ and the source of $C_k$ belong to $Y_0,Y_r$, respectively.
\end{itemize}
Then $C$ is an element of the family parametrized by $\CX$.
\end{corollary}

\begin{proof}
Consider a subscheme $C=C_1\cup\dots \cup C_k$ satisfying the properties (C1--C4), and denote by $a_{i_0}=a_0,a_{i_1},\dots, a_{i_k}=a_r$ the weights of its fixed points $P_0,\dots,P_k$. Every irreducible component $C_t$ is the closure of an orbit parametrized by the geometric quotients $\GX_{i_t,i_t+1},\dots, \GX_{i_{t+1}-1,i_{t+1}}$; let us denote the corresponding points by $g_{i_t},\dots, g_{i_{t+1}-1}$, respectively. In this way to the subscheme $C$ we can associate elements $g_{i}\in\GX_{i-1,i}$ for every $i=1,\dots,r$. Property  (C2) implies that for every $i=1,\dots,r-1$ the source of the orbit parametrized by $g_i$ coincides with the sink of the orbit parametrized by $g_{i+1}$, hence the images of this two elements into $\GX_{i,i}$ are equal, and we may conclude that they have a common preimage in $\GX_{i-1,i+1}$. Recursively, using Theorem \ref{thm:main}, we get that the $g_i$'s have a common preimage $g\in\CX_{0,r}=\CX$; in other words the image of $g\in \CX$ into $\Hilb(X)$ parametrizes a subscheme supported on $C$; since this scheme must be reduced by Lemma \ref{lem:AMFM}, it is precisely $C$. This concludes the proof.
\end{proof}

\begin{remark}\label{rem:center}
Let us describe here the blowups given by the pruning maps.
For every weight $a_k$ we will denote by $\overline Y_{k}$ the image in $\GX_{k,k}$ of the union of the fixed point components of weight $a_k$. For every $i<k<j$, we denote by $S_{i,k} \subset \CX_{i,k}$ the (scheme-theoretical) inverse image of $\overline Y_{k}$ in $\CX_{i,k}$ and by $D_{k,j} \subset \CX_{k,j}$ the inverse image of $\overline Y_{k}$ in $\CX_{k,j}$. In other words $S_{i,k}$ (resp. $D_{k,j}$) is the subvariety of $\CX_{i,k}$ (resp. $\CX_{k,j}$) parametrizing invariant $1$-cycles with source (resp. sink) in the image in $X_{i,k}$ (resp. in $X_{k,j}$) of the union of the fixed point components of weight $a_k$.

The exceptional divisor of the map $s:\CX_{i,j} \to \CX_{i,j-1}$ consists of the points parametrizing cycles meeting an inner fixed component of $X_{i,j}$ of weight $a_{j-1}$, hence 
the center of the blowup $s:\CX_{i,j} \to \CX_{i,j-1}$ is $S_{i,j-1}$; 
similarly the center of the blowup $d:\CX_{i,j} \to \CX_{i+1,j}$ is $D_{i+1,j}$.
\begin{equation}\label{eq:center}
\begin{gathered}
\xymatrix@C=1mm@R=3mm{&&&\CX_{i,j}\ar[ldd]_{\prul}\ar[rdd]^{\prur}&\\
&&&&\\
\ S_{i,j-1}\ \ar[rrddd] \ar@{^(->}[rr]& &\ \CX_{i,j-1}\ar[rdd]_{\prur}&&\CX_{i+1,j}\ar[ldd]^{\prul}&&D_{i+1,j} \ar[llddd]\ar@{_(->}[ll]\\
&&&&\\
&& &\CX_{i+1,j-1} &&\\
&&S_{i+1,j-1}\ar@{^(->}[ru] &&D_{i+1,j-1}\ar@{_(->}[lu] &
}
\end{gathered}
\end{equation}
By definition $S_{i,j-1}=d^{-1}(S_{i+1,j-1})$ and $D_{i+1,j}=s^{-1}(D_{i+1,j-1})$; they will be  smooth if $S_{i+1,j-1}$ and $D_{i+1,j-1}$ are smooth and have transversal intersection. This is not true in general, and so the Chow quotient can be singular. In the following section we will present in Example \ref{ex:BRUS} an equalized action on a smooth variety whose Chow quotient is singular.
\end{remark}

\section{Smoothness of the Chow quotient}\label{sec:smooth}

In this section we will discuss the smoothness of the normalized Chow quotient $\CX$. We will first show (Example \ref{ex:BRUS}) that $\CX$ can be singular, even in the case in which the variety $X$ is smooth and the action is equalized, that we are considering in this paper. In order to guarantee the smoothness of $\CX$, one needs to impose some positivity assumptions on $X$; we will show the smoothness of $X$ --together with the smoothness of some of its partial quotients $\CX_{i,j}$-- in the case of convex varieties (Section \ref{ssec:convex}).

\begin{example}\label{ex:BRUS}
On the projective line $\P^1$ with homogeneous coordinates $(x:y)$ set  $0:=(1:0)$, $\infty:=(0:1)$ and consider the smooth projective variety $G$ defined as the blowup of $\P^1\times\P^1\times \P^1$ along the points $(0,\infty,0)$, $(0,0,\infty)$. We denote by $\ell_-,\ell_+$ the strict transform in $G$ of the lines $\{0\}\times\P^1\times\{0\}$, $\{0\}\times\{0\}\times\P^1$, respectively, which meet at the strict transform of the point $(0,0,0)$. 
The variety $G$ admits two Atiyah flips $\psi_-$, $\psi_+$, with centers at $\ell_-$, $\ell_+$, respectively:
$$
\xymatrix@C=15mm{Y_-&G\ar@{-->}[l]_(.4){\psi_-}\ar@{-->}[r]^{\psi_+}&Y_+}
$$
We then consider two ample divisors $L_\pm$ in $Y_\pm$, identify them with two movable divisors in $G$, and consider the bi-graded finitely generated $\C$-algebra:
$$
R:=\bigoplus_{a_{\pm}\geq 0}\HH^0(G,a_-L_-+a_+L_+).
$$
Following \cite{BRUS}, we may choose $L_\pm$ and an integer grading in $R$ so that $X:=\Proj(R)$ (together with its natural polarization) is a projective variety with a $\C^*$-action of criticality $3$, and three geometric quotients:
$$
\GX_{0,1}=Y_-,\quad \GX_{1,2}=G,\quad \GX_{2,3}=Y_+.
$$
The varieties $Y_-,G,Y_+$ are smooth toric threefolds, and the birational transformations among them are compatible with the torus action; we have represented  very ample polytopes of the three varieties in question in Figure \ref{fig:example2}. 
\begin{figure}[h!!]
\includegraphics[width=12.5cm]{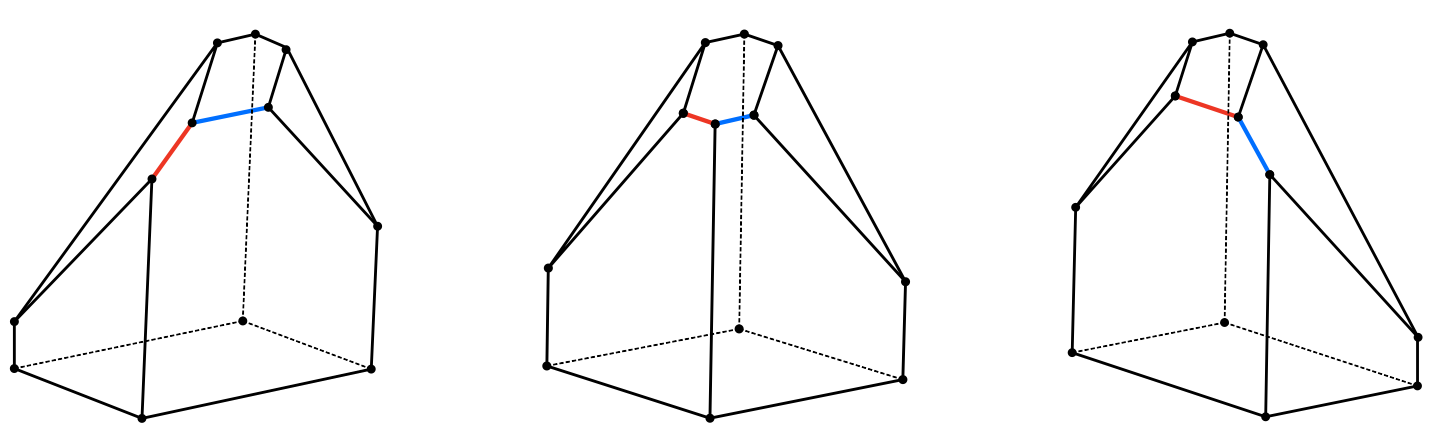}
\caption{Very ample polytopes of the varieties $Y_-,G,Y_+$.\label{fig:example2}}
\end{figure}
Then the construction of a bordism $X$  can be done in a toric way. In fact, one can check that the polytope in $\R^4=\R\otimes\Z^4$ whose vertices are the columns in the matrix below determines a smooth projective $4$-dimensional variety $X$.
\setcounter{MaxMatrixCols}{26}
\[\setlength{\arraycolsep}{1pt}
\renewcommand\arraystretch{1}
\hfsetfillcolor{gray!10}
\hfsetbordercolor{gray}
\footnotesize{\begin{pmatrix}
0 & 0 & 0 & 0 & 0 & -1 & -2 & -2 & -2 & -2 & -2 &  -2 &  \tikzmarkin{pr2}(-0.3,-0.1)(0.4,0.3) {-3}  & \tikzmarkin{pr}(-0.3,-0.1)(0.4,0.3) {-5} & 0 & 0 & 0 & 0 & 0 & -1 & -3 & -5 & -6 & -6 & -6 & -6\\
0 & 0 & -1 & -4 & -4 & 0 & 0 & 0 & -1 & -3 & -4 & -4 & -4 & -4 & 0 & 0 & -1 & -4 & -4 & 0 & -4 & -4 & 0 & 0 & -4 & -4\\
0 & -6 & 0 & -3 & -6 & 0 & -1 & -6 & 0 & 0 & -1 & -6 & 0 & 0 & 0 & -6 & 0 & -3 & -6 & 0 & 0 & 0 & -5 & -6 & -1 & -6\\
0 & 0 & 0 & 0 & 0 & 0 & 0 & 0 & 0 & 0 & 0 & 0 &  1 \tikzmarkend{pr2}& 3  \tikzmarkend{pr}& 4 & 4 & 4 & 4 & 4 & 4 & 4 & 4 & 4 & 4 & 4 & 4\end{pmatrix}}\]
In the variety $X$ we consider the $\C^*$-action corresponding to the fourth natural projection of the character lattice $\Z^4\to \Z$. With respect to the embedding determined by the polytope, this action has bandwidth $4$, the sink and the source are the varieties $Y_-$, $Y_+$, respectively, and we have two extra (inner) fixed points $P_1,P_3$ (of weights $1$,$3$, respectively, corresponding to the vertices of the polytope highlighted in the matrix above). These two points correspond to the Atiyah flips $\psi_\pm$ in the following;  the flip $\psi_-^{-1}:Y_-\to G$, for instance, substitutes the projective line parametrizing $\C^*$-orbits with source at $P_1$ (corresponding to tangent directions in $\cN^{+}(P_1)$) by the projective line parametrizing $\C^*$-orbits with source at $P_3$ (corresponding to tangent directions in $\cN^{+}(P_3)$). We have represented the $\C^*$-action on the variety $X$ in Figure \ref{fig:example2b}. 

\begin{figure}[h!!]
\begin{picture}(230,130)
\put(0,0){\includegraphics[width=8cm]{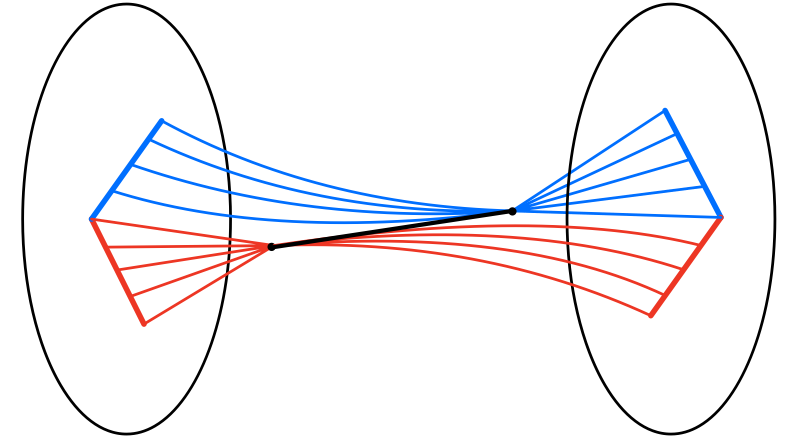}}
\put(30,16){$Y_-$}
\put(25,40){\color{red}$\ell_-$}
\put(25,80){\color{blue}$\ell_+$}
\put(185,16){$Y_+$}
\put(198,40){\color{red}$\ell_-$}
\put(200,80){\color{blue}$\ell_+$}
\put(74,42){$P_1$}
\put(138,70){$P_3$}
\end{picture}
\caption{The $\C^*$-action on the variety $X$.\label{fig:example2b}}
\end{figure}
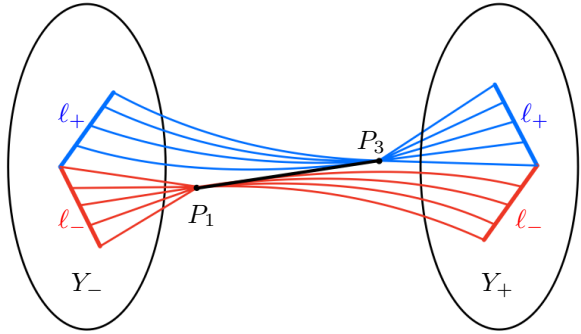

Now we note that in the sink $Y_-$ we have two distinguished rational curves, which are the intersections $B^+(P_1)\cap Y_-$ and $B^+(P_3)\cap Y_-$. Abusing notation, we denote them by $\ell_-$, $\ell_+$ (since $\ell_+$ is the strict transform of $\ell_+\subset G$ and $\ell_-\in Y_-$ is the indeterminacy locus of $\psi_-^{-1}:Y_-\to G$). The intersection of $\ell_-, \ell_+\in Y_-$ is a point. 

We claim that the Chow quotient of this action is not smooth. 
In fact, following the description presented in the proofs of Proposition \ref{prop:blowup} and Lemma \ref{lem:blowup2}, the Chow quotient $\CX$ may be constructed in two steps  as follows:
\begin{itemize}
\item The Chow quotient $\CX_{0,2}$ is the smooth blowup of $Y_-$ along $\ell_-$. 
\item The Chow quotient $\CX=\CX_{0,3}$ is the normalization of the blowup of $\CX_{0,2}$ along the inverse image of $\ell_+$. 
\end{itemize}
Since the inverse image of $\ell_+$ in $\CX_{0,2}$ consists of two $\P^1$'s meeting at a point, $\CX$ is not smooth.
\end{example}

\subsection{The case of convex varieties}\label{ssec:convex}

We finish this section by presenting a sufficient condition for the smoothness of the normalized Chow quotient $\CX$,  
which is the {\em convexity} of the variety $X$.

\begin{definition}\cite[0.4]{FP}
A smooth projective variety $X$ is said to be {\em convex} if for every morphism $\mu: \PP^1 \to X$, $H^1(\PP^1,\mu^*T_X)=0$, where $T_X$ is the tangent bundle of $X$.
\end{definition}

\begin{remark}\label{rem:convorbits}
%The convexity assumption is very restrictive. In fact, a
Among smooth rationally connected varieties, the only known convex examples are rational homogeneous spaces $G/P$, where $G$ is semisimple algebraic group and $P$ is a parabolic subgroup (see \cite{Pandharipande}). Note, however, that we will use here only ``$\C^*$-invariant convexity'', i.e. that $H^1(C,{T_X}_{|C})=0$, for closures $C$ of $1$-dimensional $\C^*$-orbits.
\end{remark}

We will now prove Theorem \ref{thm:convex}, using some ideas in \cite{FP,MMW}. The proof is divided in two steps; we will show first that the partial Chow quotients $\CX_{i,j}$, $i=0$, or $j=r$, are smooth (Proposition \ref{prop:smoothChows}), and then we will conclude by showing that the maps $s,d$ among them are smooth blowups (Proposition \ref{prop:resolution}).

\begin{proposition}\label{prop:smoothChows} Under the hypotheses of Theorem \ref{thm:main}, assume moreover that $X$ is convex. Then $\CX_{i,j}$ is smooth if $i=0$ or $j=r$. %for every $0 < j \le r$, $0 \le i <r$. 
In particular $\CX$ is smooth.
\end{proposition}

\begin{proof}
Let us prove the statement for the quotients $\CX_{0,j}$; the case of $\CX_{i,r}$ is analogous.
Note that the variety $X_{0,j}$ is by definition the pruning $X(\tau_-,\tau_+)$ for $0 < \tau_- <a_1$, $a_{j-1} < \tau_+ < j$.  

Let us consider the blowup of $X$ along the sink $Y_0$ and the source $Y_r$ of the action, $\beta: X^\flat \to X$, with exceptional divisors $Y_0^\flat$ and $Y_r^\flat$.
The indeterminacy locus of the birational modification $X_{0,j} \dashrightarrow X^\flat$ is disjoint from $Y_0^\flat$, so there exists a smooth equivariant blowdown $X_{0,j} \to X'_{0,j}$  which contracts $Y_0^\flat$ to $Y_0$. Alternatively, one may describe $X'_{0,j}$ as the pruning $X(0,\tau_+)$.
The Chow quotients $\CX_{0,j}$ and $\CX'_{0,j}$ are isomorphic (Cf. Remark \ref{rem:chblow}),
so we will show that $\CX'_{0,j}$ is smooth. To simplify notation, for the rest of the proof we will set $X':=X'_{0,j}$.
The sink of the $\C^*$-action on $X'$ is (isomorphic to) $Y_0$, and the source $Y'_r$ is a divisor  which is a birational modification of $Y_r^\flat$.

Let $C$ be the closure of a $1$-dimensional orbit of the action with sink and source $x_\pm$, and assume that $x_+\in Y'_r$. Then the weights of the $\C^*$-action on $T_{X',x_+}$ (respectively $T_{X',x_-}$) are all $0$'s and $1$'s (respectively $1$'s, $0$'s and $(-1)$'s). Then by \cite[Lemma 2.16]{WORS1} we get that the degrees of the direct summands of ${T_{X'}}_{|C}$ cannot be smaller than $-1$; thus the same holds for $\cN_{C/X'}$, and we may conclude that 
\begin{equation}\label{eq:vanish}
\HH^1(C,{T_{X'}}_{|C})=\HH^1(C,\cN_{C/X'})=0.
\end{equation}

Let us now denote by $C_g$ the closure of the general orbit of the action in $X'$, and by $[C_g]$ the corresponding point in $\Hilb(X')$; applying the above argument to $C_g$ we get that $\Hilb(X')$ is smooth at $[C_{g}]$. In particular there exists a unique irreducible component  of $\Hilb(X')$ containing it, that we denote here by $\cM$. Note that $\cM$ is a projective variety.
Let us now consider the induced $\C^*$-action on $\cM$; since $\cM$ is smooth at $[C_g]$, by \cite[Theorem 2.5]{BB} there exists a unique fixed point component  $H \subset \cM$ containing $[C_g]$. We will show that $\cM$ is smooth at every point of $H$, from which the smoothness of $H$ will follow (again by \cite[Theorem 2.5]{BB}). 

By Remark \ref{rem:Hilb}, the normalization of $H$ is $\CX'$; in particular, for every $[C]\in H$, the corresponding subscheme $C$ is a reduced tree of smooth rational curves (Lemma \ref{lem:reduced}). Furthermore, the intersection of two meeting components of $C$ is transversal, since at every fixed point $x'$ the positive and the negative part of $T_{X',x'}$ are transversal. In particular, the subscheme $C$ is a local complete intersection, hence the conormal sheaf $\cI_{C,X'}/\cI^2_{C,X'}$ is locally free, and we have a short exact sequence:
$$
\shse{\cI_{C,X'}/\cI^2_{C,X'}}{{\Omega_{X'}}_{|C}}{\Omega_C}
$$
By applying the functor $\Hom(\,\underline{\ \ }\,,\cO_C)$, we get that $\Ext^1(\cI_{C,X'}/\cI^2_{C,X'},\cO_C)$ is a quotient of $\Ext^1({\Omega_{X'}}_{|C},\cO_C)=\HH^1(C,{T_{X'}}_{|C})$. To conclude that $H$  is smooth at $[C]$ it is enough to show (see, for instance, \cite[Theorem~I.2.8]{kollar}) that \begin{equation}\label{eq:h1}
\HH^1(C,{T_{X'}}_{|C})=0.
\end{equation}
To prove (\ref{eq:h1}) we  write $C$ as $C_1 \cup C_2$, with $C_2$ irreducible and meeting the source, and set $\{p\}=C_1\cap C_2$. By (\ref{eq:vanish}), we know that $\HH^1(C,T_{X'}|_{C_2})=0$,  and so if $C=C_2$ the proof is finished. If else $C_1 \not = \emptyset$ we consider the exact sequence
\[\shse{T_{X'}|_{C_1}(-p)}{T_{X'}|_C}{T_{X'}|_{C_2}}\]
Since $X'$ is isomorphic to $X$ in a neighborhood of $C_1$, and $X$ is convex by assumption, we  have $\HH^1(C_1,T_{X'}|_{C_1}(-p))=0$ by (32) in the Proof of \cite[Lemma~10]{FP}, so  (\ref{eq:h1}) holds.
\end{proof}

\begin{proposition}\label{prop:resolution}
Under the hypotheses of Proposition \ref{prop:smoothChows}, we have a strong factorization:
$$
\xymatrix{\GX_{0,1}\ar@/^18pt/@{-->}[rrrrrr]^{\psi}&\CX_{0,2}\ar[l]^(0.42){s}&\dots \quad&\CX\ar[l]^(0.42){s}\ar[r]_(0.42){d}&\quad \dots&\CX_{r-2,r}\ar[r]_(0.42){d}&\GX_{r-1,r}}
$$
\end{proposition}

\begin{proof}
We know that the maps $s,d$ above are normalizations of blowups (Proposition \ref{prop:blowup}). Moreover the exceptional loci of these maps are divisors, whose connected components are irreducible (by Lemma \ref{lem:blowup2} applied to the prunings $X_{i,j}$, $i=0$ or $j=r$). Since for every component of the exceptional locus of each map $s$ or $d$, the dimension of the fibers is constant, it follows by \cite[Corollary~4.11]{AWDuke} that the maps $s,d$ are blowups with smooth centers.
\end{proof}
%
%\begin{remark}\label{rem:resolution}
%Under the assumptions of the above Proposition, let us consider the birational map $\psi:\GX_{0,1}\dashrightarrow \GX_{r-1,r}$ associated with the $\C^*$-action, for which --as explained in the Introduction-- $\CX$ may be considered as a resolution:
%$$
%\xymatrix{&\CX\ar[dr]\ar[dl]&\\\GX_{0,1}\ar@{-->}[rr]^{\psi}&&\GX_{r-1,r}}
%$$
%By Theorem \ref{thm:main} and Proposition \ref{prop:smoothChows} we see that $\psi$ factors as a sequence of blowups followed by a sequence of blowdowns:
%$$
%\xymatrix{\GX_{0,1}\ar@/^18pt/@{-->}[rrrrrr]^{\psi}&\CX_{0,2}\ar[l]^(0.42){s}&\dots \quad&\CX\ar[l]^(0.42){s}\ar[r]_(0.42){d}&\quad \dots&\CX_{r-2,r}\ar[r]_(0.42){d}&\GX_{r-1,r}}
%$$
%where all the intermediate varieties are smooth. This can also be rephrased by saying that the Chow quotient $\CX$ (which is relevant for some special varieties $X$ in the framework of Enumerative Geometry, see Remark \ref{rem:isoextreme} below) can be obtained from a geometric quotient by successive blowups. 
%\end{remark}

A natural problem that presents at this point is the full description of the Nef, Movable and Effective cones of $\CX$. Our results suggest that an answer can be given for the Chow quotients of a convex variety of Picard number one endowed with an equalized $\C^*$-action (see \cite{WORS3} for a description of these actions in the case of rational homogeneous spaces).

%%%%%%%%%%%%%%%%%%%%%%%%%%%%%%%%%%%%%%%%%%%%%%%%%%%%%%%%%%%%%%%%%%%%%%%%%%%%%%%
%%%%%%%%%%%%%%%%%%%%%%%%%%%%%%%%%%%%%%%%%%%%%%%%%%%%%%%%%%%%%%%%%%%%%%%%%%%%%%%

\section{Final remarks}\label{sec:final}

In Theorem \ref{thm:main} we assumed that  $\codim(B^\pm(Y),X)>1$ for every  $Y\in\cY^\circ$, or equivalently, that the induced action on $X_{0,r}$ is a bordism. However, the pruning construction allows us to extend the statement to a broader setting. For instance, we may prove the following:

\begin{corollary}\label{cor:main}
Let $(X,L)$ be a polarized pair, with $X$ smooth, endowed with an equalized nontrivial $\C^*$-action. Then the conclusions of Theorems \ref{thm:main} and \ref{thm:convex} hold if the action satisfies the following condition:
\begin{equation}\label{eq:condition}\tag{$\star$} 
\text{$\codim(B^\pm(Y),X)>1$ for every $Y\in \cY$ of weight $a_i$ with $2 \le i \le r-2$.}
\end{equation}
\end{corollary}
\begin{proof}
As already observed in Remark \ref{rem:projC2}, $\CX =\CX_{0,r}$, so we can always assume that the action is of B-type. 
If Condition (\ref{eq:condition}) is fulfilled, by Theorem \ref{thm:main}, we can assume that either $\codim(B^-(Y_1),X)=1$, or $\codim(B^+(Y_{r-1}),X)=1$.
In the first case (the second is analogous) the arguments in the proof of Lemma \ref{lem:blowup2} show that $\CX_{0,j} \simeq \CX_{1,j}$ for every $j \ge 1$. %and that $\CX_{i,r-1} \simeq \CX_{i,r}$ for every $i \le r-1$. 
In particular $\CX_{0,r} \simeq \CX_{1,r}$ and we conclude observing that $X_{1,r}$ satisfies the hypotheses of Theorem \ref{thm:main}. 
\end{proof}

\begin{remark}\label{rem:isoextreme}
Note that Condition (\ref{eq:condition}) holds, for instance, in the case in which $\Pic(X) \simeq \Z$ and the extremal fixed point components of the action are isolated (see \cite[Lemma~2.8 (2)]{WORS1}). The $\C^*$-actions of this kind on rational homogeneous spaces have been described in \cite[Section~3.4]{WORS5}). In particular, our result can be applied to the $\C^*$-actions whose Chow quotients are the classical spaces of complete collineations, complete quadrics and complete symplectic quadrics (see \cite{Thaddeus}). For these spaces a description of the Nef, Movable and Effective cones has been given in \cite{Hue,Mas1,Mas3}.\end{remark}

Another relevant example in which Condition (\ref{eq:condition}) holds is the diagonal action on  $X=(\mathbb{P}^1)^n$, whose Chow quotient $\CX$ has appeared extensively in the literature in different settings. For instance, it is isomorphic to a Losev--Manin moduli space %n+2
of pointed stable curves of genus zero (see \cite{LS}), and its associated polytope (as a toric variety) is the permutahedron of order $n$, classically known and studied in combinatorics (cf. \cite{post}). We will see that, for this example, all the intermediate Chow quotients $\CX_{i,j}$ are smooth and all the maps $s$ and $d$  appearing in Diagram \ref{fig:Chow} are blowups along smooth centers (or isomorphisms).

\begin{example}\label{ex:perm} Let $X$ be the projective variety $(\mathbb{P}^1)^n$. It  is a toric variety with an action of the torus $\TT:=(\C^*)^n$.  
The associated fan in the lattice $N$ of $1$-parameter subgroups --with $\Z$-basis $\{e_1,\dots, e_n\}$-- has $2n$ rays (generated by  $\pm e_i$). It has $2^n$ maximal cones $\langle \pm e_1, \dots, \pm e_n\rangle$ (one cone for each sequence of signs $\pm$). We take the action associated to the $1$-parameter subgroup $\nu=e_1+\cdots+e_n$, denote it by $\C^*_\nu\subset \TT$, and set $\TT':=\TT/\C^*_\nu$. Note that we have an induced action of $\TT'$ on the GIT and Chow quotients of $X$ by $\C^*_\nu$.

    We take the ample line bundle $\mathcal{O}_X(1,\dots,1)$ and the linearization of the action of $\TT$ such that the associated polytope $\Delta$ is the unit cube:  If $\{e_1^*,\dots, e_n^*\}$ is the dual basis of $\{e_1,\dots, e_n\}$, then the vertices of the cube --which are associated to $\TT$-fixed points-- are the points of the form $$e^*_I=\sum_{i\in I} e_i^*,\quad I\subseteq\{1,\dots n\}.$$
    
    The $\C^*_\nu$-action  
    is associated with a projection $ M\otimes\R\to \R$, that we denote by $\nu$ as well. The criticality of the action is $n$, with critical values $0,\dots,n$. As described in Example \ref{ex:toric}, the semi-geometric quotient $\GX_{j,j}$ associated to the critical value $j\in\{0,\dots,n\}$ is the $(n-1)$-dimensional toric variety defined by the lattice polytope $\nu^{-1}(j)\cap \Delta$, associated with an ample line bundle $\cL_{j,j}\in \Pic(\GX_{j,j})$ and a linearization of the $\C^*_\nu$-action. Note that the vertices of the polytope $\nu^{-1}(j)\cap \Delta$ are the elements of the form $e^*_I$ with $|I|=j$.  
    
Now we consider the $\TT'$-action on $\CX=\CX_{0,n}$. Notice that the orbits of the $\C^*_\nu$-action  on $X=X(\Delta)$ which are invariant with respect to the quotient torus $\TT'$ correspond to the edges of the cube $\Delta$. Therefore the $\TT'$-fixed points of $\CX_{0,n}$  correspond to trees of $\mathbb{P}^1$'s associated to a sequence of consecutive edges of the cube linking $e^*_\emptyset=0$ with $e^*_1+\dots +e^*_n$; equivalently such a tree is determined by a sequence of cube vertices $v_0=0,v_1,\dots,v_n=e^*_1+\cdots+e^*_n$ such that there exists a (unique) permutation $\sigma\in S_n$ with $$v_{n-i-1}=v_{n-i}-e^*_{\sigma(i)}\ \text{for}\ i=1,\dots,n.$$ 
Such a curve and its associated point in $\CX$ will be denoted by $C_\sigma$. 

Let us now compute the weights of the $\TT'$-action on $\CX$, with respect to a particular ample line bundle $\cL$, defined as follows; abusing notation, we denote by $\cL_{j,j}\in \Pic(\CX)$ the pullback of $\cL_{j,j}\in \Pic(\GX_{j,j})$, and set:
$$\cL:=\cL_{0,0}\otimes \dots \otimes \cL_{n,n}\in \Pic(\CX).$$
As Lemma \ref{lem:ampleg} below shows, $\cL$ is ample on $\CX$. Combining the  $\TT'$-linearizations on the bundles $\cL_{j,j}$ (and their presentations as sections of the unit cube) we get a $\TT'$-linearization on $\cL$. 
In particular, we get a linearization of the action of $\TT'$ on $\cL$, and the weight of the $\TT'$-action with respect to $\cL$ on the fixed point $C_\sigma$ is 
    $$v_n+(v_n-e^*_{\sigma(1)})+\cdots+(v_n-(e^*_{\sigma(1)}+\cdots+e^*_{\sigma(n-1)}))=\sum_{i=1}^n i\cdot e^*_{\sigma(i)}\in \Mo(\TT')\subset M.$$
   This shows that the polytope associated to the pair $(\CX,\cL)$ is the permutahedron of order $n$ (see \cite[9.2]{Ziegler}).
\end{example}

\begin{lemma}\label{lem:ampleg}
Let $(X,L)$ be a polarized pair, with $X$ smooth, endowed with an equalized nontrivial $\C^*$-action, satisfying Condition (\ref{eq:condition}).
Let $\cL_{i,i+1} \in \Pic(\CX)$ (resp. $\cL_{i,i} \in \Pic(\CX)$) be pullbacks of ample line bundles on $\GX_{i,i+1}$, (resp. on $\GX_{i,i}$).  Then the line bundles
\[
\cL_{0,1}\otimes\cdots\otimes\cL_{r-1,r}, \qquad \cL_{0,0}\otimes\cdots\otimes\cL_{r,r}
\]
are ample on $\CX=\CX_{0,r}$.
\end{lemma}

\begin{proof} 
To prove that $\cL_{0,1}\otimes\cdots\otimes\cL_{r-1,r}$ is ample
it is enough to show that the product morphism $\CX_{0,r} \to \Pi_i \GX_{i,i+1}$ is finite. This morphism maps a point $c \in \CX_{0,r}$ parametrizing a cycle
$C$ whose irreducible components  $C_1, \dots, C_s$ are closures of orbits, to a point of the form $([C_{j_0}], \dots, [C_{j_{r-1}}])$, for certain $j_0,\dots,j_s\in \{0,\dots,r\}$.
By definition $C_{j_i}$ is the (unique) irreducible component of $C$ such that the weights of $L$ at its sink and source are smaller than or equal to $a_i$, and bigger than or equal to $ a_{i+1}$, respectively.

Therefore two cycles in the inverse image of a point $c'$ in the image of $\CX_{0,r} \to \Pi_i \GX_{i,i+1}$ have the same irreducible components; by Lemma \ref{lem:reduced} they have the same image in the Chow quotient $\overline{\CX}\subset\Chow(X)$. Since $\CX$ is the normalization of $\overline{\CX}$,  the result follows.

The ampleness of $ \cL_{0,0}\otimes\cdots\otimes\cL_{r,r}$ will follow from the first part if we show that
 the product morphisms $\GX_{i,i+1}\rightarrow\GX_{i,i}\times\GX_{i+1,i+1}$ are finite, for every $i$. Let us notice that two elements $[C]$ and $[C']$ of $\GX_{i,i+1}$ are mapped to the same element of $\GX_{i,i}$ (resp. $\GX_{i+1,i+1}$) if and only if they have the same sink, with $L$-weight $a_i$ (resp.  the same source, with $L$-weight $a_{i+1}$).
Assume by contradiction that for some $i$ there is a curve $B\subset\GX_{i+i+1}$ which is contracted by $\GX_{i,i+1}\rightarrow\GX_{i,i}\times\GX_{i+1,i+1}$. Then there is a one dimensional family $\{C_b, b \in B\}$ of closures of orbits with the same sink of $L$-weight $a_i$ and the same source of  $L$-weight $a_{i+1}$, hence,  by Mori's Bend-and-Break (see \cite[Proposition 3.2]{De}) a reducible or non reduced $\C^*$-invariant cycle with sink of $L$-weight $a_i$ and source of  $L$-weight $a_{i+1}$, a contradiction.
\end{proof}

It is known (see \cite[1.1]{DRS}) that the Chow quotient $\CX$ of $X=(\P^1)^n$ by the diagonal $\C^*$-action is smooth, and that it is isomorphic to the Losev-Manin space $\overline{L}_n$ (see \cite{LS}), hence it can be constructed starting from $\P^{n-1}$ (which is a GIT quotient of $X$) via a sequence of blowups along smooth centers (see \cite[6.1]{Hassett}). We will now show that also all the intermediate Chow quotients $\CX_{i,j}$ are smooth, and that the maps $s,d$ among them are smooth blowups. In particular, $\CX$ can be constructed upon any geometric quotient $\GX_{i,i+1}$ of $X$ by a certain sequence of smooth blowups.

\begin{proposition}
For the action described in Example \ref{ex:perm}, all the intermediate Chow quotients $\CX_{i,j}$ are smooth and all the maps $s$ and $d$  appearing in Diagram (\ref{fig:Chow}) are blowups along smooth centers or isomorphisms.
\end{proposition}

\begin{proof}
In order to use inductive arguments, let us set $X^n:=(\PP^1)^n$.
After blowing up the sink and the source of the action on $X^n$ we have a B-type action on $X^n_{0,n}$ which satisfies Condition (\ref{eq:condition}); as observed at the beginning of the section we  have $\CX^n_{0,j} \simeq \CX^n_{1,j}$ for $j \ge 1$ and  $\CX^n_{i,n-1} \simeq \CX^n_{i,n}$ for $i \le n-1$ and $X^n_{1,n-1}$ satisfies the assumptions of Theorem \ref{thm:main}.

Moreove we know, by Lemma \ref{lem:blowup1}, that $s:\CX^n_{i,i+2} \to \GX^n_{i,i+1}$ and $d:\CX^n_{i,i+2} \to \GX^n_{i+1,i+2}$ are blowups of smooth varieties along smooth centers. We will then consider only the ``inner'' triangle:

\begin{equation*}
\begin{tikzcd}[
  column sep={2.7em,between origins},
  row sep={3.2em,between origins}]
&&&&&&\CX^n_{1,n-1}\arrow[rd,"\prur"] \arrow[dl,"\prul",labels=above left] &&&&&&\\
&&&&&\CX^n_{1,n-2}\arrow[rd,"\prur"] \arrow[dl,"\prul",labels=above left]&&\CX^n_{2,n-1}\arrow[rd,"\prur"] \arrow[dl,"\prul",labels=above left]&&&&&\\
&&&&\CX^n_{1,n-3} \arrow[rd,dotted,no head]\arrow[dl,"\prul",labels=above left]&&\CX^n_{2,n-2}\arrow[ld,dotted,no head]\arrow[rd,dotted,no head]&&\CX^n_{3,n-1}\arrow[rd,"\prur"]\arrow[ld,dotted,no head]&&&&\\
&&& \arrow[ld,dotted,no head]&&~ &&~&&\arrow[rd,dotted,no head]&&&\\
&&\CX^n_{1,3}&&\dots &&\dots&&\dots&&\CX^n_{n-3,n-1} &&\\
\end{tikzcd}
\end{equation*}

Since we know that the varieties in the bottom row are smooth, and that the maps $s,d$ are blowups, it will be enough to prove that the centers are smooth varieties. Note that, for $n=4$ the ``inner'' triangle consists only of the smooth element $\CX^4_{1,3}$, so the statement is true. We will  now assume that the statement holds for $X^l$, with $l \le n-1$, and prove it for $X^n$.

To describe the centers we start by describing the closure of the BB-cells of the action. Let $x_k$ be a fixed point of weight $k$; then $B^+(x_k) \simeq X^k, B^-(x_k) \simeq X^{n-k}$.
For every integer $k$ such that $i\le k \le j$, the image of $B^\pm(x_k)$ in the pruning $X^n_{i,j}$ ($i\le k \le j$) is the pruning of $B^\pm(x_k)$,  that is:
\[
 \varphi_{i,j}(B^+(x_k)) \simeq X^{k}_{i,k}, \qquad \varphi_{i,j}(B^-(x_k)) \simeq  X^{n-k}_{k,j}.
\]

Now we observe that, with the notation of Remark \ref{rem:center},  $S_{i,k}$ is the Chow quotient of the (disjoint) union over the fixed points of weight $k$ of $\varphi_{i,k}(B^+(x_k)) \simeq X^k_{i,k}$ and, similarly, $D_{k,j}$ is the Chow quotient of the (disjoint) union over the fixed points of weight $k$ of $\varphi_{k,j}(B^-(x_k)) \simeq X^k_{k,j}$, so the center of the blowup $s:\CX^n_{i,j} \to \CX^n_{i,j-1}$ is  a disjoint union of varieties isomorphic to $\CX^{j-1}_{i,j-1}$, and the center of 
 the blowup $d:\CX^n_{i,j} \to \CX^n_{i+1,j}$ is a disjoint union of varieties isomorphic to $\CX^{i+1}_{i+1,j}$.
By the inductive assumption, since $1 \le i < j \le n-1$, the centers of the blowups are smooth and the proof is concluded.
\end{proof}

\begin{remark} The geometric quotient $\GX^n_{1,2}$ is the sink of $X^n_{1,2}$, which is the projective space $\PP^{n-1}$ blown up at $n$ points. These points are the intersection of the exceptional divisor of the blow up of $X^n$ along the sink with the varieties $B^+(x_1)$, where $x_1$ is a fixed point of weight one.
It can be shown, recursively, that the blowups 
\[\CX=\CX^n_{1,n-1} \to \CX^n_{1,n-2} \to \dots \to \CX^n_{1,3} \to \GX^n_{1,2}\]
are the ones giving the Losev--Manin space $\overline L_n$ in %the construction described in 
\cite[Section~6.1]{Hassett}.
\end{remark}

%%%%%%%%%%%%%%%%%%%%%%%%%%%%%%%%%%%%%%%%%%%%%%%%%%%%%%%%%%%%%%%%%%%%%%%%%%%%%%%
%%%%%%%%%%%%%%%%%%%%%%%%%%%%%%%%%%%%%%%%%%%%%%%%%%%%%%%%%%%%%%%%%%%%%%%%%%%%%%%

\bibliographystyle{plain}
\bibliography{bibliomin}

\end{document}